\documentclass[12pt,reqno]{amsart}

\usepackage{amsmath,amstext,amssymb,amsopn,amsthm}
\usepackage{url,verbatim}
\usepackage{mathtools}
\usepackage{enumerate}

\usepackage{slashbox}

\usepackage{hyperref}

\usepackage{color,graphicx}

\usepackage[margin=30mm]{geometry}
\usepackage{eucal,mathrsfs,dsfont}%gives nicer set-names. 

\allowdisplaybreaks

\newtheorem{theorem}{Theorem}[section]
\newtheorem{corollary}[theorem]{Corollary}
\newtheorem{lemma}[theorem]{Lemma}
\newtheorem{proposition}[theorem]{Proposition}

\theoremstyle{definition}

\newtheorem{definition}[theorem]{Definition}

\newtheorem{remark}[theorem]{Remark}
\newtheorem{example}[theorem]{Example}

\theoremstyle{remark}
\newtheorem{step}{Step}
\numberwithin{equation}{section}

\newcommand{\eps}{\varepsilon}

\newcommand{\calG}{\mathcal{G}}

\newcommand{\calD}{\mathcal{D}}

\newcommand{\calV}{\mathcal{V}}

\newcommand{\calX}{\mathcal{X}}
\newcommand{\calY}{\mathcal{Y}}
\newcommand{\calZ}{\mathcal{Z}}

\newcommand{\R}{\mathds{R}}

\newcommand{\bG}{\mathbf{G}}
\newcommand{\bM}{\mathbf{M}}

\newcommand{\bS}{\mathbf{S}}
\newcommand{\bX}{\mathbf{X}}

\newcommand{\bZ}{\mathbf{Z}}

\newcommand{\Rt}{{\R^2}}
\newcommand{\prt}{\partial}

\newcommand{\wh}{\widehat}
\newcommand{\wt}{\widetilde}

\newcommand{\rpar}[1]{\left( #1 \right)}
\newcommand{\kpar}[1]{\left\{ #1 \right\}}
\newcommand{\spar}[1]{\left[ #1 \right]}

\DeclareMathOperator{\dist}{dist}

\def\bx{{\bf x}}

\newcommand\Tstrut{\rule{0pt}{2.6ex}}       % Top strut
\newcommand\Bstrut{\rule[-1.2ex]{0pt}{0pt}} % Bottom strut

\title[Number of elastic collisions]{A lower bound for the number of elastic collisions}
\author{Krzysztof Burdzy and Mauricio Duarte}

\address{KB: Department of Mathematics, Box 354350, University of Washington, Seattle, WA 98195}
\email{burdzy@uw.edu}

\address{MD: Departamento de Matematicas, Universidad Andres Bello. Republica 498, Santiago, Chile.} 
\email{mauricio.duarte@unab.cl}

\thanks{KB's research was supported in part by Simons Foundation Grant 506732. MD was supported by Proyecto FONDECYT 11160591,  N\'ucleo Milenio NC130062, and Basal CONICYT Program PFB 03.}

\begin{document}

\begin{abstract}
We prove by example that the  number of elastic collisions of $n$ balls
of equal mass and equal size
in $d$-dimensional space can be greater than $n^3/27$ for $n\geq 3$ and $d\geq 2$. The previously known lower bound was of order $n^2$.
\end{abstract}

\maketitle

\section{Introduction}\label{intro}

Let $K(n,d)$ be the supremum of the number of elastic collisions of $n$ balls  of equal radii and masses in $d$-dimensional space; the supremum is taken over all initial conditions (positions and velocities).  
The rigorous definition of a ``collision'' will be given in Section \ref{prelim}; see especially Remark \ref{m19.1}. Here we limit ourselves to the informal remark that, in this paper, we consider only collisions involving pairs of balls in which both velocities change by a non-zero amount. Our examples do not involve simultaneous collisions (see Remarks \ref{re:ja2.1b} and \ref{oc20.3}).

It is obvious that $K(2,d)=1$ for all $d\geq 1$. It was proved in \cite{Vaser79} that $K(n,d)< \infty$ for all $n$ and $d$.

For $n\geq 3$, let $n_1 = \lfloor n/3 \rfloor$, $n_2 = n - 2 n_1$, and
\begin{align}\label{d31.1}
f(n)=n_1(n_1 +1) n_2 + n_2(n_2-1)/2 + n_1(n_1-1).
\end{align}
The following is our main result.

\begin{theorem}\label{oc20.1}
For all $n\geq 3$ and $d\geq 2$, 
\begin{align}\label{oc15.1}
K(n,d) \geq f(n).
\end{align}
\end{theorem}

 It is elementary to check that
\begin{align}\label{oc15.2}
&\lim_{n\to \infty} \frac{f(n) }{ n^3/27} = 1,\\
& f(n) > n^3/27 \qquad \text {  for  } n \geq 3,\label{ja2.1a}\\
& f(n) > n(n-1)/2 \qquad \text {  for  } n \geq 7,\label{oc15.3}\\
& f(n) = n(n-1)/2 \qquad \text {  for  } n =6,\label{oc15.4}\\
& f(n) < n(n-1)/2 \qquad \text {  for  } 3\leq n \leq 5.\label{oc15.44}
\end{align}

A well known elementary argument, recalled in Example \ref{oc18.2} below,
shows that   $K(n,1) =n(n-1)/2$ 
for $n\geq 2$. It is obvious that $K(n,d_1) \geq K(n,d_2)$ for all $n$ and $d_1\geq d_2$. Hence, 
\begin{align}\label{ja14.1}
K(n,d) \geq n(n-1)/2 \qquad \text{  for all } n\geq 2,\ d\geq 1.
\end{align}
Intuition may suggest that the bound in \eqref{ja14.1} is sharp because the balls are ``most constrained'' in one dimension; see \cite{MurCoh} for a historical review related to this point. 
It turns out that this intuition is wrong. It is known that $K(3,2) = 4 > 3(3-1)/2$. An example showing that $K(3,2) \geq 4$ was found by J.D.~Foch and published in \cite{MurCoh}. The proof that $K(3,2) < 5$
was given in \cite{MC93}.
We are not aware of any values of $n\geq 4$ and $d\geq 2$ for which it is already known that $K(n,d) >n(n-1)/2$.
Hence, in view of \eqref{oc15.3}, our lower bound given in \eqref{oc15.1} is the first result of this type. Because of \eqref{oc15.4}-\eqref{oc15.44}, Theorem \ref{oc20.1} leaves the intriguing possibility that the bound in \eqref{ja14.1} is sharp for some $4 \leq n \leq 6$ and  $d\geq 2$. To settle this question, we will prove the following result.

\begin{theorem}\label{n6.1}
For all $4 \leq n \leq 6$
and $d\geq 2$,
\begin{align*}
K(n,d) \geq K(n,2) \geq 1+ n(n-1)/2.
\end{align*}
\end{theorem}

Our proof of Theorem \ref{n6.1} actually works for all $n\geq 4$ but it gives a result weaker than Theorem \ref{oc20.1} for $n\geq 7$.

The example of Foch published in \cite{MurCoh}, our Theorem \ref{oc20.1}, bound 
\eqref{oc15.3} and  Theorem \ref{n6.1} imply that the elementary bound \eqref{ja14.1} is never sharp in higher dimensions. More precisely, we have the following result.
\begin{corollary}
For all $ n \geq 3$ and $d\geq 2$,
\begin{align*}
K(n,d) > n(n-1)/2.
\end{align*}
\end{corollary}

\medskip
We tried several natural ideas to improve the bound in  \eqref{oc15.1}
but none of them worked; see Remark \ref{oc18.1} for details. 

The proof of Theorem \ref{oc20.1} is partly based on the ideas behind the ``pinned billiards balls'' model, to be discussed in \cite{ABD}. In the pinned billiards balls model, touching static balls are associated with vectors (``velocities'') and vectors corresponding to  adjacent balls  change according to the rules normally applied to velocities of  colliding  moving balls. 

\medskip

The question of whether a finite system of hard balls can have an infinite number of collisions was posed by Ya.~Sinai. It was answered in negative in \cite{Vaser79}. For alternative proofs see \cite{Illner89, Illner90,IllnerChen}.  The papers  \cite{BFK5,BFK2, BFK1,BFK3, BFK4} were the first to present universal upper bounds for the number of collisions of $n$ hard balls in any dimension.
The first of these bounds for $K(n,d)$ was given in  \cite{BFK1},
namely
\begin{align}\label{s26.1}
K(n,d) \leq
\left( 32  n^{3/2}\right)^{n^2}.
\end{align}
The following alternative  bound  appeared in \cite{BFK5},
\begin{align}\label{s26.2}
K(n,d)\leq
\left( 400  
 n^2\right)^{2n^4}.
\end{align}
We stated simplified versions of the original bounds, because
the original versions allowed for unequal radii or masses. 

\section{Preliminaries }\label{prelim}

\subsection{General notation.} We will use $| \,\cdot\,|$ to denote the usual Euclidean norm. We will use the following notation: $a\land b = \min(a,b)$, $a\lor b = \max(a,b)$.

We will denote the left and right limits by ``$-$'' and ``$+$'', for example, $v_j(t-) = \lim_{s\uparrow t} v_j(s) $. We will write  $D x(t)$ to denote the right derivative of the function $x$ at time $t$. Henceforth, whenever we say ``derivative,'' we mean the ``right derivative.'' 

Given $w\in\Rt\setminus\kpar{(0,0)}$ and $z\in\Rt$, we define the ``projection'' $P_w( z) = P_w z = (z\cdot w)w/|w|^2$.

\subsection{Elastic collisions of balls.} 
\label{m5.1}
We will consider $n\geq2$ hard spheres in $\R^d$, for $d\geq 1$, colliding elastically. In some cases, the evolution will be restricted to the time interval $[0, \infty)$ and in some other cases the time interval will be $(-\infty,\infty)$.

We will assume that all balls have radii equal to 1 and identical masses.
Let us consider only two balls $X$ and $Y·$  with centers given by $x(t)$ and $y(t)$ at time $t$. We will say that balls $X$ and $Y$ collide at time $t$ if $|x(t) - y(t)| = 2$ and their velocities change at this time. The velocities are constant between collision times. The norm of the velocity will be called speed. The velocity of the center of a ball is well defined at all times except at collision times (there are finitely many collision times in our setting). Nonetheless, the right derivative of the center of a ball is well defined at all times. Thus, the velocity of ball $X$ will be denoted by $D  x$.

If only two balls collide at time $t$, then the laws of conservation of energy and momentum determine  the velocities  after the collision. 
Suppose that balls $X$ and $Y$ collide at time $t$ and no other ball touches any of these two balls at time $t$. 
This can happen only if  $D x(t-) $ and $ D y(t-)$  satisfy
\begin{align}\label{oc4.3}
(D x(t-) - D y(t-)) \cdot (x(t) - y(t)) < 0.
\end{align}
Let $\bx = x(t) - y(t)$.
Then the velocities just after the collision are given by
\begin{align}
\label{oc2.3}
D  x(t) &= D  x(t-) - P_{\bx} \rpar{ D  (x - y)(t-) } ,\\
\label{oc4.2}
D  y(t) &= D  y(t-) - P_{\bx} \rpar{ D  (y - x)(t-) } .
\end{align}
In other words, the balls exchange the components of their velocities that are parallel to the line through their centers at the moment of impact. The orthogonal components of velocities remain unchanged.

\begin{remark}\label{re:ja2.1b}

Crucially for our arguments, if only two balls are involved in a collision then their trajectories restricted to any finite time interval are continuous functions of the initial conditions (positions and velocities) in the topology of uniform convergence.  This claim follows easily from the explicit equations \eqref{oc2.3}-\eqref{oc4.2}. Moreover, the claim extends to joint continuity of any finite number of billiard balls as long as there are no simultaneous collisions, in the sense outlined below. 

We say that a ``simultaneous collision'' occurs at time $t$ if there is a collection of balls $\{B_1, B_2, \dots, B_k\}$, $k\geq3$, such that for any two balls $B_i$ and $B_j$ in the family, there exist $j_1=i, j_2, \dots,j_{m-1}, j_m=j$ such that $B_{j_r}$ is in contact with $B_{j_{r+1}}$ at time $t$ for all $r=1, \dots, m-1$. 

If a ball $B_1$ touches $B_2$ at time $t$ and balls $B_3$ and $B_4$ also touch at time $t$ but none of the balls from the first pair touches a ball from the second pair, we do not call $t$ a simultaneous collision time. This type of simultaneous occurrence of two collisions does not present any technical difficulties. From the point of view  of counting collisions, if there are no simultaneous collisions in the sense given above, one can modify the initial conditions slightly and the modified system will have the same number of collisions, none of them occurring simultaneously with any other.

Our examples will involve configurations coming very close to simultaneous collisions but there will be no simultaneous collisions. See \cite{Vaser79,IllnerChen} for the analysis of the trajectories of families of balls allowing for simultaneous collisions.

We note parenthetically that
a billiards trajectory in a polyhedral domain is a continuous function of initial conditions if and only if every angle between two faces on the boundary of the domain has the form $\pi/m$ for some integer $m\geq 1$ (see \cite[Thm.~1, p.~22]{KozTre}).

\end{remark}

\begin{remark}\label{m19.1}

Since our main results, Theorems \ref{oc20.1} and  \ref{n6.1}, give lower bounds for the numbers of collisions, we want to stress that we count only ``uncontroversial'' collisions. Specifically, we say that there is a collision at a space-time point $(y,t)$ if  and only if (i) two balls are tangent at $y$ at time $t$, (ii) for some $\eps>0$, the two balls do not touch in intervals $(t-\eps,t)$ and $(t, t+\eps)$, and (iii) both balls change velocities at time $t$. 

The following events are not counted as collisions in Theorems \ref{oc20.1} and  \ref{n6.1}. First, there are no simultaneous collisions in the sense of Remark \ref{re:ja2.1b} in the evolutions of ball families constructed in the proofs of the two theorems.

Second, we do not count ``grazing collisions,'' i.e., points in space-time when (i) two balls are tangent at $y$ at time $t$,  and (ii) the balls do not change velocities at time $t$. 

Third, our examples proving Theorems \ref{oc20.1} and  \ref{n6.1} do not involve balls traveling together, i.e., there is no pair of balls  such that the two balls touch at every time in the interval $[s,t]$ for some $s<t$.

\end{remark}

\begin{remark}\label{oc20.3}
Since our examples involve trajectories that nearly miss simultaneous collisions, we will briefly review this topic.
The laws of physics (the conservation of energy, momentum and angular momentum) do not uniquely determine (in general) the velocities after a simultaneous collision, i.e., a collision that involves more then two balls at the same time. The following example illustrates the point. 
Consider discs $A$, $B$ and $C$, with initial positions and velocities as follows: 
\begin{align*}
&x_A(0) = (0,-\sqrt{3}), \qquad\quad x_B(0) = (-1,0), \quad\quad x_C(0) = ( 1,0), \\
&D  x_A(0) = (0,\sqrt{3}), \qquad D  x_B(0) = (1,0), \qquad D  x_C(0) = (-1,0).
\end{align*}
These conditions describe a simultaneous collision. We will modify slightly the positions, but not the velocities in the following way: move $B$ slightly to the left, and $C$ slightly to the right, whereas $A$ is moved down and to the left so $A$ is closer to $B$ than to $C$. Disc $A$ is dislocated more than $B$ and $C$, so that the first collision after $t=0$ involves $B$ and $C$ only. At this collision $B$ and $C$ interchange velocities. The next collision will involve $A$ and $B$, and then there will be another one involving $A$ and $C$. If the initial displacements are of order $\eps$, then straightforward computation shows that the only possible collisions are the ones we just described, and after they take place, the final velocities are:
\begin{align*}
D  x_A =   \rpar{\frac14 , \frac{\sqrt{3}}{4}} + O(\eps), \quad D  x_B = \rpar{ -\frac32 , \frac{\sqrt{3}}{2} } + O(\eps), \quad D  x_C =\rpar{ \frac54 , \frac{\sqrt{3}}{4} } + O(\eps).
\end{align*}
Note that $B$ and $C$ have significantly different terminal velocities. If we start all over, but move $A$ down and to the right so it gets closer to $C$ than to $B$, the final velocities in this case can be obtained from the ones given  above by applying the symmetry with respect to the vertical axis. Thus, different initial positions arbitrarily close to a simultaneous collision yield very different outcomes for velocities.
In other words, in general, billiards trajectories are not continuous functions of  initial conditions. 
\end{remark}

\begin{example}\label{oc18.2}
We present a brief discussion of the one-dimensional case as it is an important ingredient in our main example. It is easy to see that in the one-dimensional case the radii of balls (i.e., the lengths of one-dimensional rods) play no essential role in the analysis of the evolution of the system  so we can and will replace the balls with reflecting points in this example. By doing this, we can focus our analysis on the gaps between neighboring balls.

We want to analyze the evolution of  $n$ reflecting points, whose positions at time $t$ are denoted $x_1(t), x_2(t), \dots, x_n(t)$.

Consider a collection of $n$ non-interacting points with the same initial positions and  velocities as  the $n$ reflecting points. Let the positions of points in the non-colliding system be denoted $\wt x_k(t)$. Thus $\wt x_k(t) = \wt x_k(0) + t D \wt x_k(0) =  x_k(0) + t D  x_k(0)$ for all $t\geq 0$ and $k$.
The slanted half-lines in Fig.~\ref{fig4} represent the trajectories of  non-interacting points.
\begin{figure}[!h] 
\includegraphics[width=0.71\linewidth]{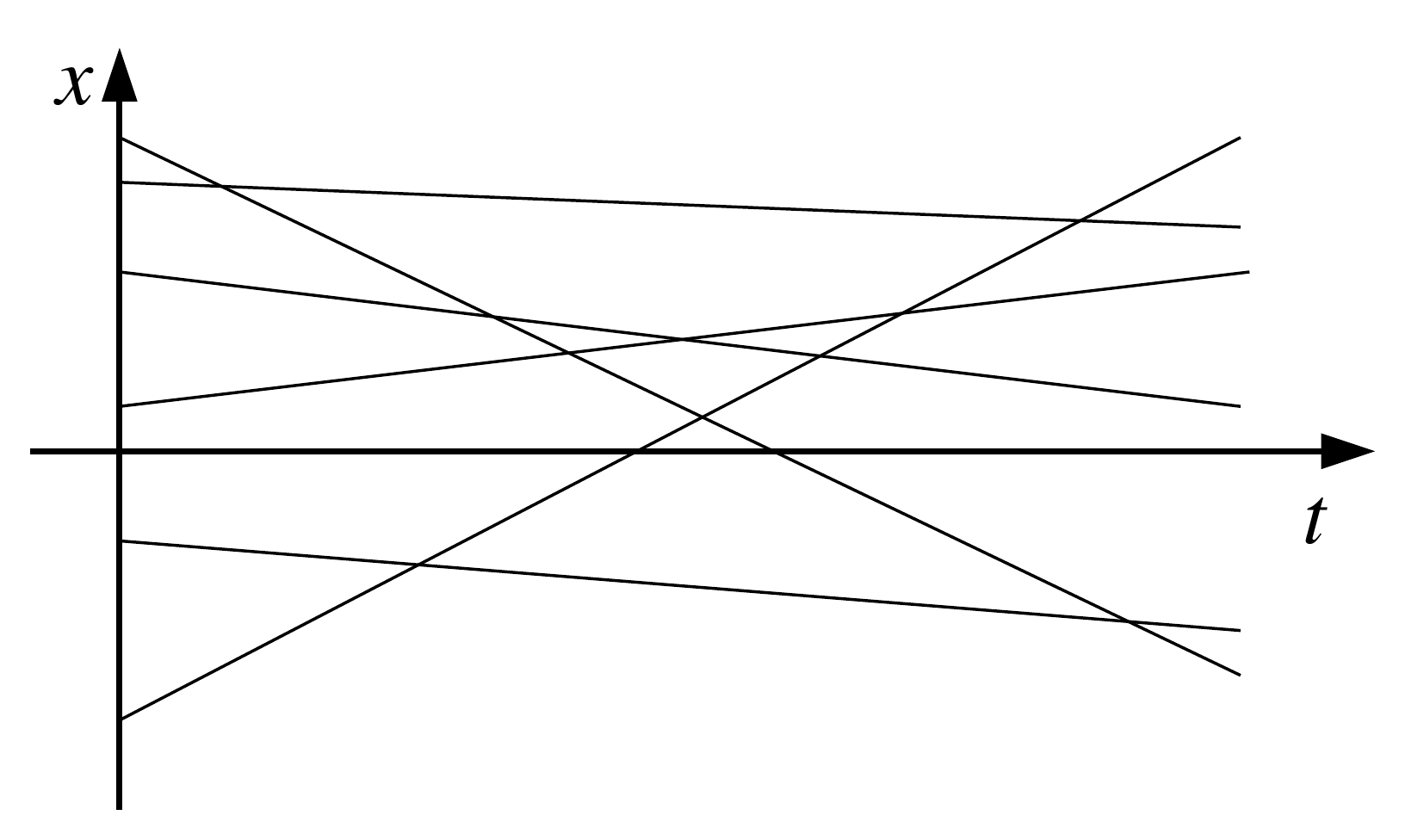}
\caption{
The figure represents trajectories (straight half-lines) of six non-interacting points moving along a line at constant velocities.
}
\label{fig4}
\end{figure}

It is well known and easy to check that the ``same'' picture represents reflecting points. If we let $\wh x_1(t)\leq \wh x_2(t) \leq \dots \leq \wh x_n(t)$ denote the ordering of $\{\wt x_1(t),\wt x_2(t), \dots , \wt x_n(t)\}$ then every function $\{\wh x_k(t), t\geq 0\}$ represents 
the trajectory of one of the points in the original system with collisions, i.e., $\{\wh x_k(t), t\geq 0\}=\{ x_j(t), t\geq 0\}$  where $j$ is such that $x_j(0) = \wh x_k(0)$.
Fig.~\ref{fig5} shows $\{\wh x_3(t), t\geq 0\}$
 in red color.
 
\begin{figure}[!h] 
\includegraphics[width=0.71\linewidth]{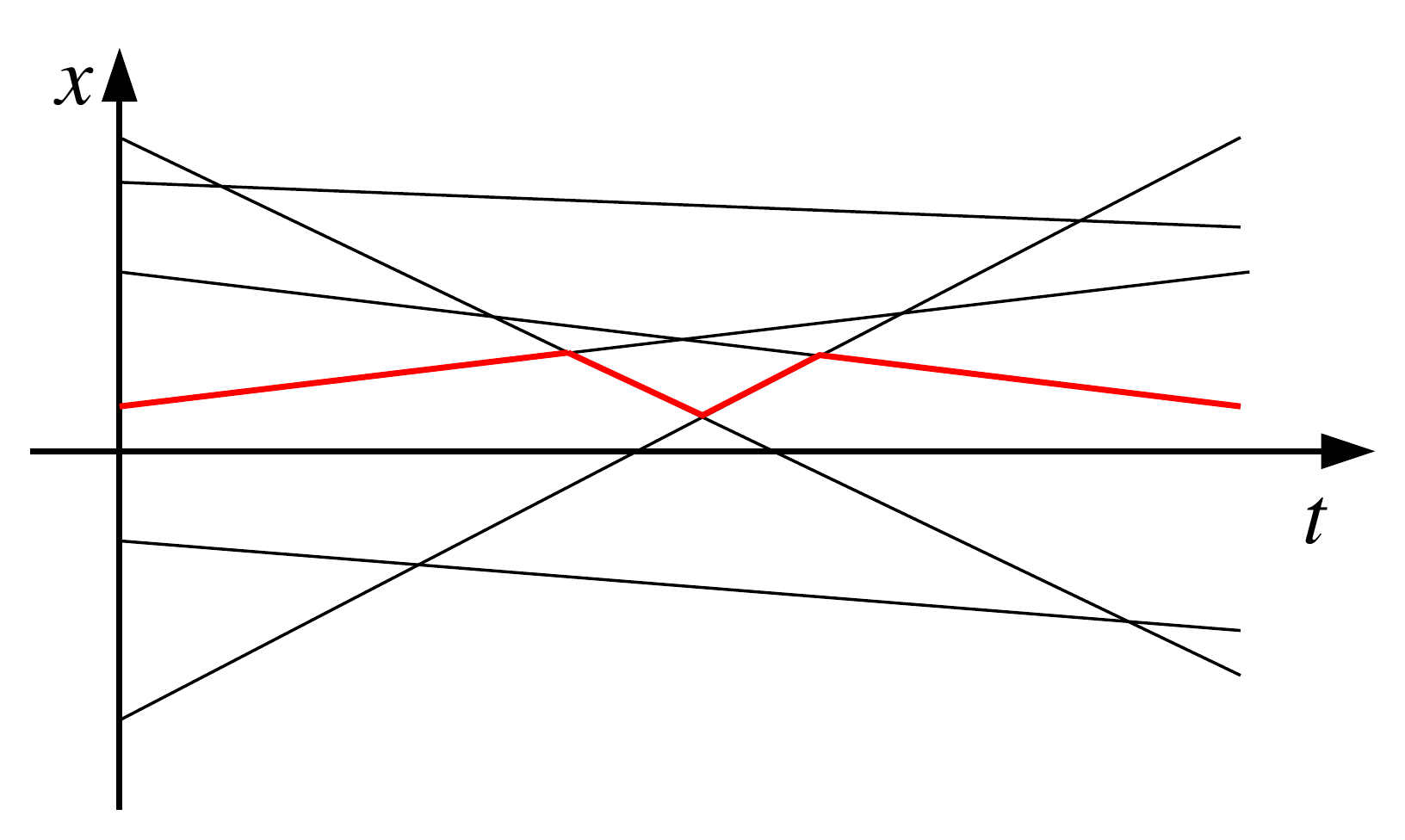}
\caption{
The same union of half-lines as in Fig.~\ref{fig4}, properly relabeled, represents trajectories of colliding points. The red polygonal line
is one of the reflecting trajectories.
}
\label{fig5}
\end{figure}
It is now clear that the maximum number of collisions in the one-dimensional system is not greater than the number of intersections of $n$ half-lines. The number of intersections of (distinct) half-lines is bounded by the number of pairs of half-lines, i.e., $n(n-1)/2$. 
To see that this bound is attained, let $x_k(0) = k$  for $k=1,\dots, n$ and $D  x_1(0) = 0$. Then use induction. Once $D  x_k(0)$ are chosen for $k=1,\dots,m$, find $D  x_{m+1}(0)< \min_{1\leq k \leq m} D  x_k(0)$ such that the half-line $\{x_{m+1}(0) + t D  x_{m+1}(0), t\geq 0\}$  intersects all  half-lines $\{x_k(0) + t D  x_k(0), t\geq 0\}$, $k=1,\dots,m$, but does not create any simultaneous intersections.
\end{example}

\section{The main example}\label{mainex}

\subsection{Notation.} 
\label{se:notex}
We define the following vectors 
\begin{align*}%\label{n1.1}
w_0 &= (0,1), & u_0 &= (1,0), \\
w_1& =(-\sqrt{3}/2, 1/2),  &w_2&=(\sqrt{3}/2, 1/2),\\
u_1& = (1/2, \sqrt{3}/2), &u_2 &= (-1/2, \sqrt{3}/2).
\end{align*}
Note that for $k=0,1,2$, vectors $w_k$ and $u_k$ are orthogonal, and each one of them has unit length.  For $k=0,1,2$, we set $L_k = \{z\in \Rt: z = c w_k \text{  for some  } c\in \R\}$.

Recall that $P_w( z) = P_w z = (z\cdot w)w/|w|^2$ for $z,w\in\R^d$, $w\ne0$. Given an integer $n>0$, we will write
$n_1 = \lfloor n/3 \rfloor$ and $n_2 =
n-2n_1$, so $2n_1 + n_2 = n$.
We will divide a family of $n$ discs into three subfamilies 
$\{A_1, A_2, \dots, A_{n_2}\}$,  $\{B_1, B_2, \dots, B_{n_1}\}$
and $\{C_1, C_2, \dots, C_{n_1}\}$.
Let $a_k(t),b_k(t),$ and $c_k(t)$ denote the positions of the centers of discs $A_k, B_k,$ and $C_k$ at time $t$, respectively. In order to simplify some formulas, we set  $b_0(t)=a_1(t)$ and $c_0(t)=a_1(t)$, although there are no discs $B_0$ and $C_0$.

\subsection{Qualitative  description of the evolution}
\label{ja15.1}
This section contains  an informal description of our main example.
The description is  idealized in the sense that the positions and velocities of the discs in the fully rigorous example will be slightly different from these in the present informal version.
The initial positions of the discs are shown in Fig.~\ref{fig1}.

\begin{figure}[!b] 
\includegraphics[width=0.5\linewidth]{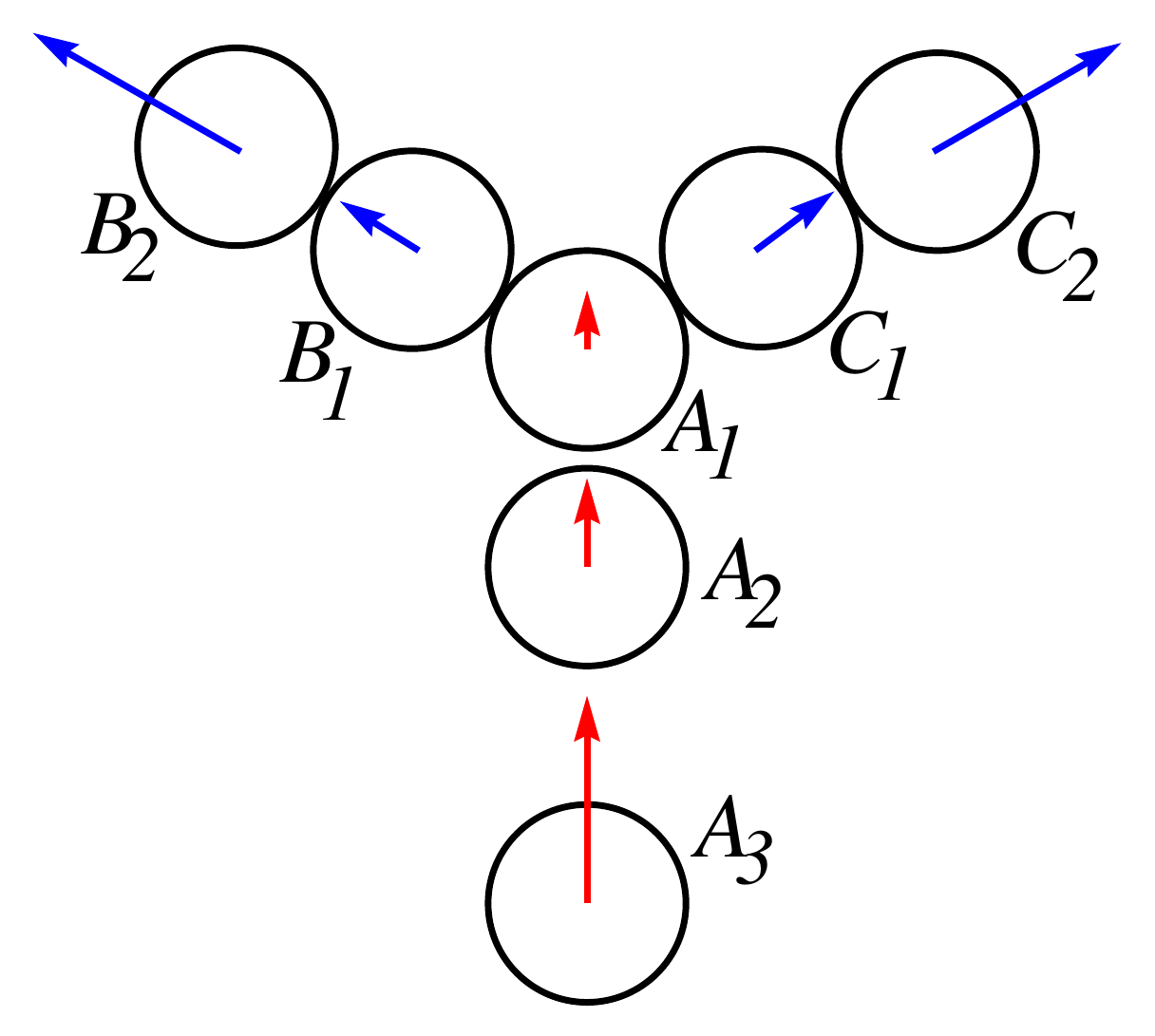}
\caption{
The figure represents the initial positions of the discs in the case $n=7$, $n_1=2$ and $n_2=3$. No pair of discs
are touching each other but some pairs of discs are depicted as touching  because the gaps between them
are extremely small. The red vertical arrows, anchored at the centers of the  three discs along the vertical arm, represent their
initial velocities. 
All other initial velocities are zero.
The blue arrows, at angles $30$ and $150$ degrees from the horizontal, represent the
terminal velocities of the four discs in the upper ``arms.'' 
}
\label{fig1}
\end{figure}

More formally, we place the center of disc $A_1$ at the origin, i.e., $a_1(0)=(0,0)$.
The centers of discs $B_k$ are on the line $L_1$ and the centers of discs $C_k$ are on the line $L_2$. The discs are arranged in the following order along the first line: $A_1, B_1, B_2, \dots, B_{n_1}$. Similarly, the  discs are arranged in the following order along $L_2$: $A_1, C_1, C_2, \dots, C_{n_1}$. On each of the two lines, the discs are positioned very close to one another, although none of the discs touches any other one.
The discs $A_1, A_2, \dots, A_{n_2}$ are placed along the negative part of the vertical axis, in this order. The  distances $|a_{k+1}(0)-a_k(0)|$ are very large and grow rapidly with $k$.

The initial velocities of $ B_1, B_2, \dots, B_{n_1}, C_1, C_2, \dots, C_{n_1}$ are zero. The velocity vectors of $A_1, A_2, \dots, A_{n_2}$ point in the upward direction, i.e., $D  a_k(0)/|D  a_k(0)| = (0,1)$ for all $k$. The speeds $|D  a_k(0)|$ are rapidly increasing as functions of $k$. 

Next we will describe the evolution of the system. At the first stage of the evolution, disc $A_1$ will hit discs $B_1, C_1, B_1, C_1 ...$. 
In other words, it will repeatedly hit these discs, alternating between them.
The total number of hits of $B_1$ will be $n_1$, and the total number of hits of $C_1$ will be also $n_1$. The first time disc $B_1$ is hit, there will be many collisions between discs $B_1, B_2, \dots, B_{n_1}$, with the result that $B_{n_1}$ will acquire substantial velocity and all other discs $B_1, B_2, \dots, B_{n_1-1}$ will have negligible velocities (see Fig.~\ref{fig6}).

\begin{figure}[!b] 
\includegraphics[width=0.8\linewidth]{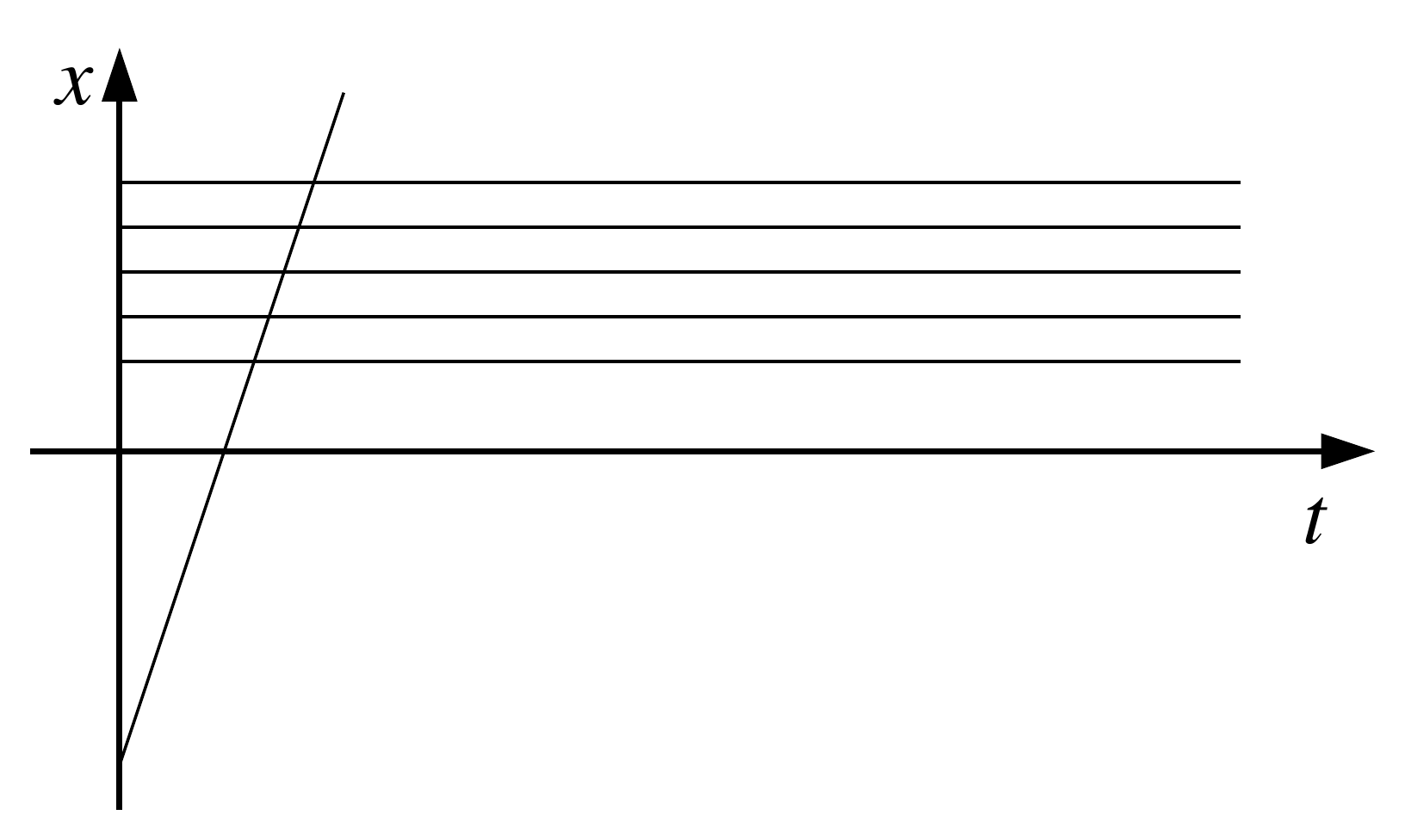}
\caption{
The figure represents trajectories of colliding particles on a line.  Five particles have zero initial velocities while one particle, at one end of the configuration, has a very large initial velocity. Only one particle, at the other end of the configuration, has a large terminal velocity.
}
\label{fig6}
\end{figure}

The second time disc $B_1$ is hit, many collisions between discs $B_1, B_2, \dots, B_{n_1-1}$ will occur, with the result that $B_{n_1-1}$ will acquire substantial velocity and all other discs $B_1, B_2, \dots, B_{n_1-2}$ will have negligible velocities. In general, after the $k$-th collision of $A_1$ with  $B_1$, there will be collisions between discs $B_1, B_2, \dots, B_{n_1-k+1}$, with the result that $B_{n_1-k+1}$ will acquire substantial velocity and all other discs $B_1, B_2, \dots, B_{n_1-k}$ will have negligible velocities. An analogous sequence of collisions will occur on the other side, involving discs $C_k$. The total number of collisions among $A_1$ and $B_k$'s at this stage of the evolution will be $n_1(n_1+1)/2$, and the same count applies to the collisions among $A_1$ and $C_k$'s. Hence, the total number of collisions involving $A_1$, $B_k$'s and $C_k$'s will be $n_1(n_1+1)$ at this stage of the evolution.

The initial distances between consecutive discs $B_{n_1}, \dots, B_1, A_1, C_1, \dots, C_{n_1}$ are assumed to have been so small that the distances between the discs at the end of the evolution described so far will be 
negligibly small, compared to the disc radius, as well.

At the second stage of the evolution, disc $A_2$ will arrive and it will hit $A_1$. The initial position and velocity of $A_2$ are chosen so that it will arrive at a time when the first stage of the evolution is over but the discs $B_{n_1}, \dots, B_1, A_1, C_1, \dots, C_{n_1}$ are still very close to their initial positions. The velocity of $A_2$ is assumed to be so high that velocities of  
$B_{n_1}, \dots, B_1, A_1, C_1, \dots, C_{n_1}$ acquired during the first stage of the evolution are negligibly small, compared to $|D a_2(0)|$. The disc $A_1$ will acquire a velocity very close to that of $A_2$ and the sequence of collisions that occurred during the first stage of evolution will occur again, at a much higher speed. Hence, the number of collisions 
between $A_1$, $B_k$'s and $C_k$'s will be $n_1(n_1+1)$ at the second stage of the evolution.  

We proceed by induction. Disc $A_k$ will arrive after the $(k-1)$-th stage of the evolution is over, it will start a sequence of collisions between consecutive discs $A_{k-1}, A_{k-2}, \dots, A_1$, with the result that $A_1$ will acquire  a velocity that will dwarf all previously occurring velocities and will result in $n_1(n_1+1)$ collisions 
between $A_1$, $B_k$'s and $C_k$'s  at this  stage of the evolution. 
Since the pattern will be repeated $n_2$ times,  the number of collisions 
between $A_1$, $B_k$'s and $C_k$'s, summed over all stages of the evolution, will be $n_1(n_1+1)n_2$.  

We add the $(n_2-1)n_2/2$ collisions between $A_k$'s to obtain the first two terms in \eqref{d31.1}:
\begin{align*}
n_1(n_1+1)n_2 + (n_2-1)n_2/2.
\end{align*}

The last term in \eqref{d31.1} represents the ``preparation'' of the initial conditions. At the beginning of this section, we assumed that discs $B_k$ were  stationary at the initial time $t=0$. We now change this assumption and instead assume that $B_k$'s have undergone $n_1(n_1-1)/2$ collisions (the maximum possible in one dimension) before reaching the ``initial'' positions at time $t=0$. The same assumption is made about discs $C_k$, so the total contribution to \eqref{d31.1} is $n_1(n_1-1)$, i.e., the last term in that formula. But then the discs $B_k$ and $C_k$ have non-zero initial velocities. This does not change the count of the collisions after time $t=0$ because the speeds of $B_k$'s and $C_k$'s at time $t=0$ can be made arbitrarily small and  we can invoke the continuity of the trajectories as functions of the initial conditions (see Remark \ref{re:ja2.1b}).

\begin{remark}\label{oc18.1}
None of the following three potential ways to   improve the bound in Theorem \ref{oc20.1} seems to work. 

\begin{enumerate}[(i)]
\item In the current version, about $n/3$ discs are assigned to each of the three arms of the initial configuration (see Fig.~\ref{fig1}). Every other proportion of discs in the three arms yields a lower number of collisions.

\item In higher dimensions, one could place a larger number of ``arms''  next to the ball $A_1$. This modification would not increase the number of collisions. 
  
\item Discs in the upper arms could be arranged in a more complicated pattern, say, a tree-like structure. Once again,  this would not increase the number of collisions.
\end{enumerate}
\end{remark}

\subsection{Pinned billiard balls}
\label{d7.1}

This section contains an alternative  informal description of the evolution of the system of discs. The description given in Section \ref{ja15.1} was totally qualitative. We will give some explicit (although approximate) velocities in this section. We hope that this will help the reader to follow the formal, very technical, proof of Theorem \ref{oc20.1}.

One of the main ideas behind the proof of Theorem \ref{oc20.1} comes from the pinned billiard balls model, to be discussed in greater detail in \cite{ABD}.
In the pinned billiard balls model, balls do not move at all. Some balls touch some other balls. Each ball has an associated vector that can be thought of as pseudo-velocity. A sequence of
pairs of touching balls is chosen by an external process for pseudo-collisions.
At the time of each pseudo-collision, the pseudo-velocities associated with the pair of touching balls change values as in \eqref{oc2.3}-\eqref{oc4.2}.

We will now represent the first stage of the evolution (in which discs $A_2,\ldots,A_{n_2}$ are not involved) of our main example as a pinned billiard balls model. All gaps described in Section \ref{ja15.1} as small  are now assumed to be zero. Discs $B_k$ lie next to one another with their centers on $L_1,$ and discs $C_k$ lie next to one another with their centers on $L_2$. Initially, only $A_1$ has non zero velocity $w_0=(0,1)$.  We choose  $B_1$ and $A_1$ to collide first, which yields velocities $P_{w_1}w_0$ for $B_1$, and $P_{u_1}w_0$ for $A_1$, after the collision.  

Next we choose   $B_1$ and $B_2$ to collide, then $B_2$ and $B_3$, and so on, until $B_{n_1}$  acquires velocity $P_{w_1}w_0$, and all other discs $B_k$ have velocity zero. Then, we choose  $A_1$ to collide with $C_1$, which leaves $A_1$ with velocity  $P_{u_2}P_{u_1}w_0$, and $C_1$ with velocity  $P_{w_2}P_{u_1}w_0$. Then $C_1$ collides with $C_2$, then $C_2$ with $C_3$, and so on, until $C_{n_1}$ acquires velocity $P_{w_2}P_{u_1}w_0$, all other discs $C_k$ have velocity zero. At this point, there are only three discs with non zero velocities: disc $B_{n_1}$ has velocity $P_{w_1}w_0$, disc $C_{n_1}$ has velocity $P_{w_2}P_{u_1}w_0$, and disc $A_1$ has velocity $P_{u_2}P_{u_1}w_0$. Therefore, the only possible collision is between $A_1$ and $B_1$. After this collision, disc $A_1$ has velocity $P_{u_1}P_{u_2}P_{u_1}w_0$, and disc $B_1$ has velocity $P_{w_1}P_{u_2}P_{u_1}w_0$. This last velocity is transmitted by collisions to disc $B_{n_1-1}$, but its magnitude is not sufficiently large  to allow for a collision between $B_{n_1-1}$ and $B_{n_1}$, according to \eqref{oc2.3}-\eqref{oc4.2}. 
The process is continued inductively.

After the $k$-th collision between $A_1$ and $B_1$, disc $A_1$ is headed towards $C_1$ with velocity
\begin{equation}
\label{ja3.2}
 (P_{u_1}P_{u_2})^{k-1}P_{u_1}w_0 = \frac{1}{2^{2k-2}}	 (w_0\cdot u_1)u_1 =  \frac{\sqrt{3}}{2^{2k-1}}	 u_1 = \frac{3}{2^{2k}}w_2 + \frac{\sqrt{3}}{2^{2k}}u_2.
\end{equation}
After the $k$-th collision between $A_1$ and $C_1$, disc $A_1$ is headed towards $B_1$ with velocity 
\begin{equation}
\label{ja3.1}
 (P_{u_2}P_{u_1})^{k}w_0 = \frac{1}{2^{2k-1}}	(w_0\cdot u_1)u_2 =
\frac{\sqrt{3}}{2^{2k}} u_2 = \frac{3}{2^{2k+1}}w_1 + \frac{\sqrt{3}}{2^{2k+1}}u_1.
\end{equation}

The main technical challenge in the  proof of Theorem \ref{oc20.1}
is to show that a certain sequence of disc evolutions converges, in an appropriate sense, to the pinned billiard balls model outlined above. The reason for the great complexity of that argument is that our configuration of the pinned balls represents simultaneous collisions and, therefore, we cannot appeal to continuity of trajectories as functions of initial conditions (see Remark \ref{oc20.3}).

\subsection{The Skorohod space.}\label{Skor}
A detailed discussion of the Skorohod space can be found in \cite[Ch.~3, Sect.~14]{Bill}.

Let $\calD[0, T] = \calD([0,T] ; \R^m)$ denote the set of all functions $f:[0,T]\to \R^m$ which have left limits and are continuous on the right at every point in $[0,T]$.
The set $\calD[0,\infty)$ is defined in a similar manner. Each of these sets is variably referred to as the Skorokhod space, the space of RCLL functions or the space of c\`adl\`ag functions.

We will define two metrics on $\calD[0,T]$. Let $\Lambda_T$ be the family of all strictly increasing continuous mappings of $[0,T]$ onto itself. For $f,g \in D[0,T]$, we define $\dist^T(f,g)$ as the infimum of positive $\eps$ such that there exists $\lambda \in \Lambda_T$ satisfying the following conditions,
\begin{align*}
\sup_{0\leq t \leq T} |\lambda(t) - t| \leq \eps \qquad \text{  and  }\qquad
\sup_{0\leq t \leq T} |f(t) -g(\lambda( t))| \leq \eps.
\end{align*}

For $\lambda \in \Lambda_T$, let 
\begin{align*}
\|\lambda\|_T = \sup_{0\leq s < t \leq T} \left | \log
\frac{\lambda(s) - \lambda(t)}{s-t} \right|.
\end{align*}
For $f,g \in \calD[0,T]$, we define $\dist_0^T(f,g)$ as the infimum of positive $\eps$ such that for some $\lambda \in \Lambda_T$,
\begin{align*}
 \|\lambda(t)\|_T \leq \eps \qquad \text{  and  }\qquad
\sup_{0\leq t \leq T} |f(t) -g(\lambda( t))| \leq \eps.
\end{align*}
The metrics $\dist^T$ and $\dist_0^T$ are equivalent, i.e., they generate the same topology. The first of these metrics is not complete but the latter one is.

\section{Proof of the main theorem} 

We  proceed to show that the evolution informally described in Sections \ref{ja15.1} and \ref{d7.1} can actually occur. Recall that in the informal description, the evolution was divided into stages. At any of these stages, one of the discs $A_k$ 
arrived with a great velocity and caused a large number of collisions among the discs that participated in the earlier stages of the evolution. We will analyze one of these stages, corresponding to the arrival of a disc labeled $A_m$. The  initial strategy is to project trajectories of disc centers onto lines $L_k$ so that the problem becomes one-dimensional, in a suitable sense. Then we will express the information about the motion of the centers in terms of the evolution of the gaps---this transformation is similar to the one in Example  \ref{oc18.2}, where we replaced one-dimensional balls (rods) with reflecting points. 
For our example to work, the gaps between the discs have to be very small.
The gaps will be rescaled so that they are of magnitude 1. Finally, we will show that the ``rescaled gap processes'' converge to a non-degenerate limit when the initial magnitudes of the gaps converge to 0, and satisfy appropriate conditions. 

\subsection{Rescaling of the system of discs} 
\label{se:system}
Recall notation from Sections \ref{se:notex} and \ref{ja15.1}. Fix some $m,n_1\geq 1$. We will represent a system of $2n_1+m$ discs as a function
\begin{align}\label{f10.1}
{\bf S}(t) = \rpar{ a_m(t),\ldots, a_1(t),b_1(t),\ldots, b_{n_1}(t),c_1(t), \ldots, c_{n_1}(t)},
\end{align}
where each component represents  the trajectory of the center of a disc of radius $r=1$ in $\Rt$. 
Functions $a_k,b_k$ and $c_k$ correspond to discs $A_k, B_k$ and $C_k$, resp.
The time derivative (velocity) of each of these functions is assumed to be well defined at all times, except for the finitely many times when collisions occur. Recall that $D$ stands for the right continuous version of the derivative with respect to $t$.
It will be convenient to define $b_0(t) $ and $c_0(t)$ as $a_1(t)$, i.e., 
  $b_0(t)=c_0(t)=a_1(t)$. These artificial functions $b_0(t) $ and $c_0(t)$, not representing any discs, will simplify some notation.

\begin{definition}
\label{de:gaps}
Let
\begin{align*}
\Gamma^{A}_k(t) &= w_0 \cdot ( a_{k-1}(t) - a_{k}(t) ) -2, \qquad k=2,\ldots,m , \\
\Gamma^B_k(t) &= w_1 \cdot ( b_{k}(t) - b_{k-1}(t) ) -2 , \qquad k=1,\ldots,n_1 , \\
\Gamma^C_k(t) &= w_2 \cdot ( c_{k}(t) - c_{k-1}(t) ) -2 ,\qquad k=1,\ldots,n_1 , \\
\rho^-_\Gamma(t) &= \min_{2\leq j\leq m}\Gamma^{A}_j(t)
\land \min_{1\leq j\leq n_1}\Gamma_j^{B}(t) \land \min_{1\leq j\leq n_1}\Gamma_j^{C}(t),\\
\rho^+_\Gamma(t) &= \max_{2\leq j\leq m}\Gamma^{A}_j(t)
\lor \max_{1\leq j\leq n_1}\Gamma_j^{B}(t) \lor \max_{1\leq j\leq n_1}\Gamma_j^{C}(t).
\end{align*}
\end{definition}

Somewhat informally speaking,
the functions $\Gamma^A_k, \Gamma^B_k$ and $\Gamma^C_k$ represent gaps between the discs projected on lines $L_j$.

\begin{definition}\label{de:ic}
For $\eps>0,\ \rho>1$, we say that a system $\bS$ of $2n_1+m$ discs satisfies $(\eps,\rho)$ initial conditions at time $t_0\in\R$ if:
\begin{enumerate}[(i)]
\item We have that $w_1\cdot a_1(t_0)\geq 0$, $w_2\cdot a_1(t_0)\geq 0$, and $w_0\cdot a_1(t_0)\leq \eps$.
\item For $j=0,\ldots, n_1$, and $k\geq 2,\ldots,m$, it holds that 
\begin{align}
\label{eq:ic_u}
| P_{u_0} a_k(t_0)| \vee | P_{u_1}b_j(t_0) |  \vee | P_{u_2} c_j(t_0) | \leq \eps.
\end{align}

\item It holds that
\begin{align}
\label{eq:ic_gap}
\eps\rho^{-1} \leq \rho^-_\Gamma(t_0) \quad\text{and}\quad \rho^+_\Gamma(t_0)\leq \eps.
\end{align}

\item It holds that either
\begin{align}
\label{eq:ic_ratio}
\frac{w_1\cdot ( b_{1}(t_0) - a_{1}(t_0) ) - 2 }{w_2\cdot ( c
_{1}(t_0) - a_{1}(t_0) ) - 2 }  \leq \frac 23, \quad\text{or}\quad \frac{w_2\cdot ( c
_{1}(t_0) - a_{1}(t_0) ) - 2 }{w_1\cdot ( b_{1}(t_0) - a_{1}(t_0) ) - 2 }  \leq \frac 23.
\end{align}
\end{enumerate}
\end{definition}

The meaning of the following definition is that  we replace discs with point masses, as in Example  \ref{oc18.2}, and enlarge the distances by the factor of $\eps^{-1}$, while slowing down the evolution at the same time, to keep the energy uniformly bounded.

\begin{definition}
We set
\begin{align}
\label{eq:X0}
&\calX^{A,\eps}_1 (t) = \eps^{-1}w_0\cdot a_1(\eps t), \quad 
\calX^{B,\eps}_0(t) = \eps^{-1}w_1\cdot a_1(\eps t), \quad 
\calX^{C,\eps}_0(t) = \eps^{-1}w_2\cdot a_1(\eps t).
\end{align}
Note that $\calX^{B,\eps}_0(t) + \calX^{C,\eps}_0(t) =  \calX^{A,\eps}_1(t)$.
For $j=1,\ldots,n_1$, and $k=2,\ldots m$, we recursively define
\begin{align}
\label{eq:XA}
&\calX^{A,\eps}_k (t) = \calX^{A,\eps}_{k-1}(t) - \eps^{-1} \Gamma^A_{k}(\eps t) , \\
\label{eq:XB}
&\calX^{B,\eps}_j (t) = \calX^{B,\eps}_{j-1}(t) + \eps^{-1} \Gamma^B_{j}(\eps t) , \\
\label{eq:XC}
&\calX^{C,\eps}_j (t) = \calX^{C,\eps}_{j-1}(t) + \eps^{-1} \Gamma^C_j(\eps t) .
\end{align}
The associated vector of scaled point masses corresponding to  $\bf S$ is defined by
\begin{align}
\label{eq:bX}
\bX^\eps(t) = \rpar{ \calX^{A,\eps}_m(t),\ldots, \calX^{A,\eps}_{1}(t),
\calX^{B,\eps}_0(t),\ldots, \calX^{B,\eps}_{n_1}(t), 
\calX^{C,\eps}_0(t),\ldots, \calX^{C,\eps}_{n_1}(t)}.
\end{align}
\end{definition}

\begin{remark}
\label{re:gap_dd}
(i) The functions in \eqref{eq:bX} represent positions of point masses associated to the respective discs, but they do not precisely follow the dynamics of elastic collisions. 

(ii)
By induction, we can show that for $k=1,\ldots,m$, and $j=0,\ldots, n_1$,
\begin{align}
\label{eq:dA}
&D \calX^{A,\eps}_k(t) = w_0\cdot D  a_k(\eps t), \\
\label{eq:dB}
&D  \calX_j^{B,\eps}(t) = w_1\cdot D  b_j(\eps t) , \\
\label{eq:dC}
&D  \calX_j^{C,\eps}(t) = w_2\cdot D  c_j(\eps t).
\end{align}
As long as only the $2n_1+m$ discs represented by $\bS$ are involved in the collisions, the speed of  $\bS$ (as a moving point in $\R^{2n_1+m}$) is constant. It follows easily from \eqref{eq:dA}-\eqref{eq:dC}  that $| D  \bX^\eps(t) | \leq 3 | D  \bS(t_0) |$ 
\begin{align}\label{f10.3}
| D  \bX^\eps(t) | \leq 3 | D  \bS(t_0) |, \qquad t\geq t_0.
\end{align}
 The factor 3 in the inequality is due to the presence of 
extra components $\calX^{B,\eps}_0(t)$ and $\calX^{C,\eps}_0(t)$ in $\bX^\eps(t)$, with no counterparts in $\bS$.
\end{remark}

Let $\bM$ be the family of all functions $f: [0, \infty) \to \R^{2+2n_1}$ such that each coordinate is a Lipschitz function with Lipschitz constant 2, i.e., $|f_j(s) - f_j(t)| \leq 2|s-t|$ for every $f=(f_1,f_2,\dots, f_{2+2n_1})\in\bM$, $1\leq j \leq 2 + 2n_1$ and $s,t\geq 0$.
For $f, h\in \bM$, we define the distance between $f$ and $h$ as
\begin{align}\label{d26.2}
\dist(f,h) = \sum_{k=1}^\infty 
2^{-k} \left(1 \land \sup_{0\leq t \leq k} |f(t) - h(t)| \right).
\end{align}
The topology associated with this metric is  the topology of uniform convergence on compact time intervals (``time'' refers here to the domain of functions in $\bM$).

\subsection{The limiting evolution}
\label{de:z}
In this section, we will define a family of functions that can be limit points for evolutions of systems defined  in \eqref{eq:bX} as the initial size of all gaps goes to zero. The definition will be complicated because, at the heuristic level, the data are the initial values of the gaps between reflecting points. Then we will construct  trajectories of the reflecting points and finally  we will go back and define time evolutions of gap processes. The functions corresponding to $A_1$ will require a separate and different treatment since disc $A_1$ hits disc $B_1$ along $L_1$ and disc $C_1$ along $L_2$. The functions constructed in this step do not have a direct interpretation as trajectories of colliding discs.

For any $\rho>1$, we will define a compact subset $\bG_\rho$ of $\bM$.
Consider non-negative real numbers $\calZ^B_{0}(0)$, $ \calZ^C_{0}(0) $, $\calZ^A_1(0) = \calZ^{B}_0(0) + \calZ^{C}_0(0)$, $\calG^A_k(0)$, $\calG^{B}_j(0) $, and $\calG^{C}_j(0) $ for $k=2,\ldots,m$, and $j=1,2,\dots,n_1$.
Suppose that these numbers also satisfy 
\begin{align}
\label{j5.01}
\calG^{B}_1(0) \leq \frac23 \calG^{C}_1(0),
\end{align} 
\begin{align}\label{d2.1}
\rho^{-1}\leq \min_{1\leq j \leq n_1} \calG^{B}_j(0) \land
\min_{1\leq j \leq n_1} \calG^{C}_j(0) \land \min_{2\leq k \leq m} \calG^{A}_k(0),
\end{align}
and
\begin{align}\label{n26.10}
\calZ^A_1(0)\lor \max_{2\leq k \leq m} \calG^{A}_k(0)
\lor \max_{1\leq j \leq n_1} \calG^{B}_j(0) \lor
\max_{1\leq j \leq n_1} \calG^{C}_j(0)
\leq 1.
\end{align}
For $k=2,\ldots,m$, and $j=1,2,\dots,n_1$, let
\begin{align}\notag
\calZ^A_k(0) &= \calZ^A_{1}(0) - \calG^{A}_2(0) - \dots - \calG^{A}_k(0),\\
\calZ^B_j(0) &= \calZ^B_{0}(0) + \calG^{B}_1(0) + \dots + \calG^{B}_j(0),
\label{f14.1} \\
\calZ^C_j(0) &= \calZ^C_{0}(0) +\calG^{C}_1(0) + \dots + \calG^{C}_j(0).\notag
\end{align}
Define times 
\begin{align}\label{n26.12}
0=t^A_m<t^A_{m-1} < \dots<t^A_1 < t^B_1 <  t^C_1< t^B_2 < t^C_2< \dots < t^B_{n_1} < t^C_{n_1},
\end{align}
by
\begin{align}
t^A_m=0, \quad t^A_{k-1} - t^A_{k} =  \calG^A_{k}(0), \quad k=2,\ldots,m,
\end{align}
\begin{align}\label{n26.13}
t^B_1 - t^A_1& = 2 \calG^{B}_1(0)  ,\\
t^C_1 - t^B_1 &= \frac43\left(\calG^{C}_1(0)  - \calG^{B}_1(0)\right),\\
t^B_j - t^C_{j-1} &= \frac{2^{2j-1}}{3} \calG^{B}_j(0), \quad j=2, \dots, n_1,\label{n26.14}\\
t^C_j - t^B_j &= \frac{2^{2j}}{3} \calG^{C}_j(0), \quad j=2, \dots, n_1.\label{n26.15}
\end{align}
The intuitive meaning of the above definition is the following. The quantities $\calG^A_k$, $\calG^B_j$ and $\calG^C_j$ represent gaps between discs. The times $t^A_k$ in \eqref{n26.12} represent the collision times between pairs $(A_{k+1},A_k)$, and the remaining  times represent collision times between $A_1$ and either $B_1$ or $C_1$. These times are defined by dividing distances (gaps) by velocities in
\eqref{ja3.2}-\eqref{ja3.1} that have been properly projected on $w_1$ or $w_2$, according to whether the next collision is with $B_1$ or with $C_1$.

Let
\begin{equation}
\label{n26.16}
\begin{aligned}
\calV^{B}_{1} &= \frac12 ,\quad \calV^C_1 = \frac34, \\ 
 \calV^{B}_{j} &= \frac{3}{2^{2j-1}}, \quad \calV^{C}_{j} = \frac{3}{2^{2j}}, 
\qquad j=2,\dots, n_1.
\end{aligned}
\end{equation}
For $t\geq 0$, let
\begin{equation}
\label{eq:ya}
\begin{aligned}
\calY^A_m(t) &= \calZ^{A}_{m}(0) +  t, \\
\calY^A_k(t) &= \calZ^{A}_{k}(0), \quad k=1,\ldots,m-1, \\
\end{aligned}
\end{equation}

Let
\begin{align}\label{j11.1a}
\left(\calZ^A_m(t), \calZ^A_{m-1}(t), \dots, \calZ^A_{1}(t)\right)
\end{align}
be the increasing ordering of $\left(\calY^A_m(t), \calY^A_{m-1}(t), \dots, \calY^A_{1}(t)\right)$ for every $t \geq 0$. Later, we will actually redefine $\calZ^A_1(t)$ for $t> t^A_1$.

For $j=1,\dots, n_1$, let 
\begin{align}
\label{eq:yb}
\calY^B_j(t) =
\begin{cases} 
\calZ^{B}_{j}(0) ,
& t\in [0, t^B_j),\\
 \calZ^{B}_{j}(0) + \calV^{B}_{j} \cdot (t-t_j^B),
& t\in [ t^B_j, \infty).
\end{cases}
\end{align}
Let
\begin{align}\label{d4.2}
\left(\calZ^B_1(t), \calZ^B_2(t), \dots, \calZ^B_{n_1}(t)\right)
\end{align}
be the increasing ordering of $\left(\calY^B_1(t), \calY^B_2(t), \dots, \calY^B_{n_1}(t)\right)$ for every $t\geq 0$.

Similarly, for $j=1,\dots, n_1$, let 
\begin{align}
\label{eq:yc}
\calY^C_j(t) =
\begin{cases} 
\calZ^{C}_{j}(0) ,
& t\in [0, t^C_j),\\
 \calZ^{C}_{j}(0) + \calV^{C}_{j} \cdot (t-t^C_j),
& t\in [ t^C_j, \infty).
\end{cases}
\end{align}
Let
\begin{align}\label{d4.3}
\left(\calZ^C_1(t), \calZ^C_2(t), \dots, \calZ^C_{n_1}(t)\right)
\end{align}
be the increasing ordering of $\left(\calY^C_1(t), \calY^C_2(t), \dots, \calY^C_{n_1}(t)\right)$ for every $t\geq 0$.

The increasing order is taken in the definition above to mimic the fact that one-dimensional collisions between discs can be treated as crossings of straight lines (see Figures \ref{fig4},\ref{fig5} and \ref{fig6}).

Let
\begin{align}
% Dynamics of B
\label{j11.2}
\calZ^B_{0}(t) &= \calZ^B_{0}(0)   \quad \text{  for  }t\in[0,t^A_1), \\
\label{ja8.3}
\calZ^B_{0}(t) &= \calZ^B_{0}(0)  + \calV^B_1 \cdot t  \quad \text{  for  }t\in[t^A_1,t^B_{1}), \\
\label{ja8.5}
\calZ^B_{0}(t)
&=
\begin{cases}
\calZ^B_1(t^B_j) & \text{  for  } t\in[t^B_j, t^C_j), \ j=1,\dots, n_1,\\
\calZ^B_1(t^B_j) + \calV^B_{j+1} \cdot (t-t^C_j) & \text{  for  } t \in [t^C_j, t^B_{j+1}), \ j=1,\dots, n_1-1,
\end{cases} \\
\label{ja8.6}
\calZ^B_{0}(t)
&= \calZ^B_1(t^B_{n_1}) + 3 (1/2)^{2n_1+1} \cdot  (t-t^C_{n_1}) 
\quad \text{  for  }t\in[t^C_{n_1}, \infty),
\\
%
% Dynamics of C
%
\label{j11.3}
\calZ^C_{0}(t) &= \calZ^C_{0}(0)  \quad \text{  for  }t\in[0,t^A_{1}), \\
\label{ja8.7}
\calZ^C_{0}(t) &= \calZ^C_{0}(0)  + \calV^B_1 \cdot t \quad \text{  for  } t \in [t^A_1,t^B_{1}), \\
\label{ja8.8}
\calZ^C_{0}(t) &= \calZ^C_{0}( t^B_1 ) + \calV^C_1 \cdot (t-t^B_{1}) \quad \text{  for  } t \in [t^B_{1}, t^C_1), \\
\label{ja8.10}
\calZ^C_{0}(t) &=
\begin{cases}
\calZ^C_1(t^C_j) & \text{  for  } t \in [t^C_j, t^B_{j+1}), \ j=1,\dots, n_1-1,\\
\calZ^C_1(t^C_j) + \calV^C_{j+1} \cdot (t-t^B_{j+1})  & \text{  for  } t \in [ t^B_{j+1},t^C_{j+1}), \ 
j=1,\dots, n_1-1,
\end{cases} \\
\calZ^C_{0}(t)
&= \calZ^C_1(t^C_{n_1}) \quad  \text{  for  }t\in[t^C_{n_1}, \infty).\label{ja8.11}
\end{align}

We next redefine $\calZ^A_1(t)$, to change its behavior for $t\geq t^A_1$, due to collisions between $A_1$ and $B_1$ or $C_1$:
\begin{align}
\label{j11.4}
\calZ^A_1(t) &= \calZ^B_0(t) + \calZ^C_0(t),\quad \text{ for } t\geq 0.
\end{align}
 The purpose of including $\calZ^A_1$ in  \eqref{j11.1a} was to have $\calZ^A_2(t) = \calZ^A_1(0)$ for $t\geq t^A_1$.

At the intuitive level, $\calZ^B_j$'s represent trajectories of $n_1$ points moving on a line and reflecting elastically, while $\calG^B_j$'s represent the gaps (distances) between consecutive points  (see Fig.~\ref{fig10}).
\begin{figure}[!b] \includegraphics[width=0.8\linewidth]{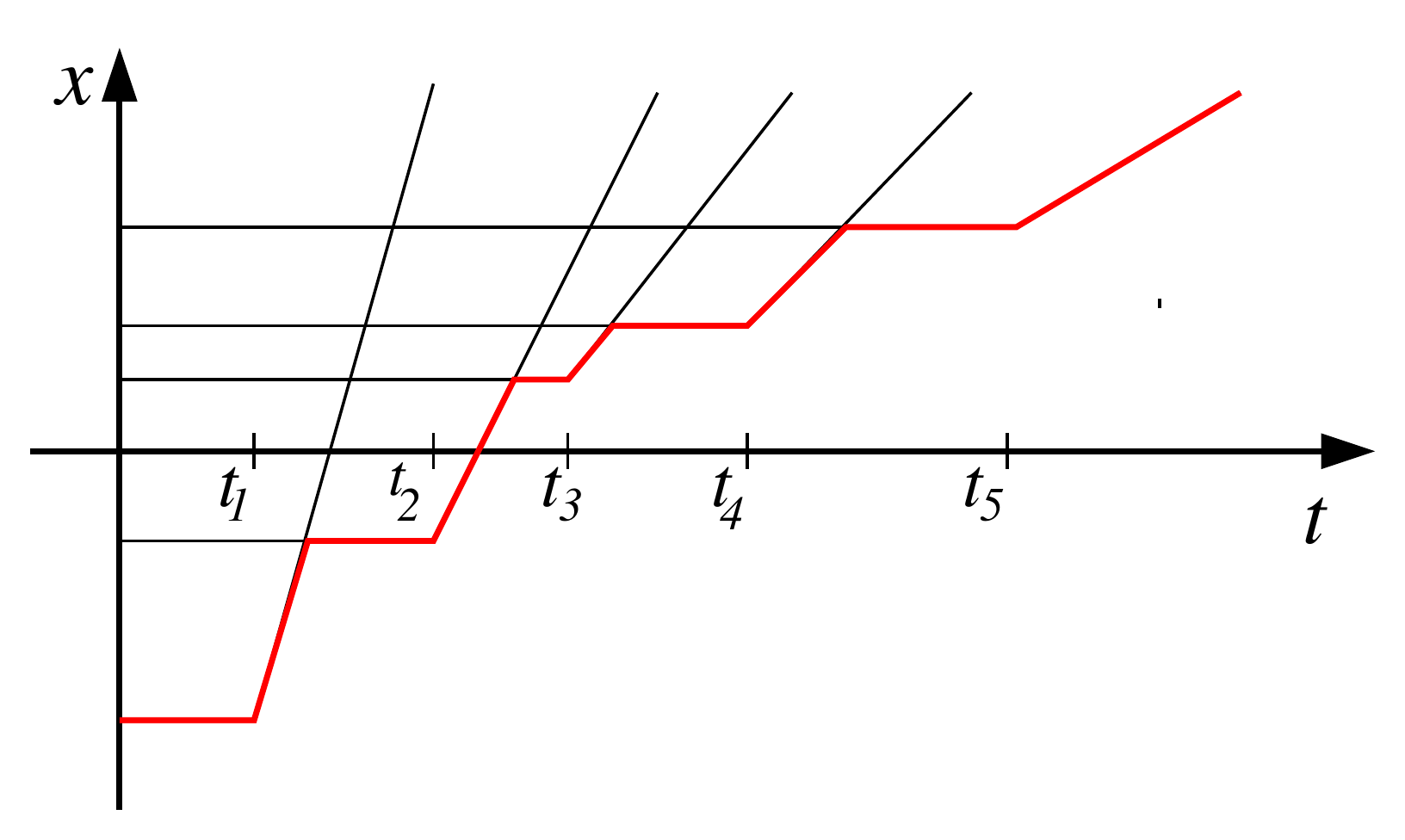}
\caption{
The figure represents the trajectories $\calZ^B_j$ of reflecting points on  line $L_1$. 
The horizontal axis represents time and the vertical axis represents  line $L_1$.
The distances between the trajectories represent the gaps $\calG^B_j$ between the reflecting points. The trajectory  $\calZ^B_1$ is marked red. The function $\calZ^B_{0}$ corresponding to $A_1$ is not shown because  the center of $A_1$  is not constrained to a line. For this reason, the sharp turns of $\calZ^B_1$ at times $t_k$ appear to have no cause---in fact, they represent collisions between $A_1$ and $B_1$.
}
\label{fig10}
\end{figure}
 The point represented by $\calZ^B_1$ is hit at times $t^B_k$. A similar remark applies to objects indexed by $C$. 

Let
\begin{align}
\label{eq:nz}
\bZ(t)
=
\left( \calZ^A_m(t), \ldots , \calZ^A_1(t),\calZ^B_{0}(t), \ldots, \calZ^B_{n_1}(t),
\calZ^C_{0}(t), \ldots, \calZ^C_{n_1}(t)
\right).
\end{align}

Let $\bG_\rho$ be the family of functions $\bZ(t)$ satisfying  conditions \eqref{j5.01}-\eqref{eq:nz}  (the dependence on $\rho$ occurs \eqref{d2.1}).  Discontinuities of $D \bZ$ occur at times $t^B_j$ and $t^C_j$, and at those positive times (hence, not at time $t^A_m=0$) when functions of type $\calY$ exchange order. It is easy to check that these times are continuous functions of the initial condition $\bZ(0)$. 

It is easy to check that $\bG_\rho \subset \bM$, i.e., all coordinates of all functions in $\bG_\rho$ are Lipschitz with the Lipschitz constant 1.
For fixed $\rho$ and $n_1$, 
the set of initial conditions of functions in $\bG_\rho$ is  compact  in view of \eqref{d2.1} and \eqref{n26.10}. These two observations imply that,
for fixed $\rho$ and $n_1$, the family $\bG_\rho$ is  a compact subset of $\bM$.

\begin{remark}
\label{re:vz}
We will now list a number of properties of functions involved in the definition of $\bZ$. We hope that our remarks will help the reader understand all the steps of the long and complicated definition.

(i) The function $\{\bZ(t), t\geq 0\}$ is completely determined by the initial condition $\bZ(0)$.

\medskip

(ii) We claim that $|D\bZ(t)|$ is uniformly bounded in time. All coordinates of $\bZ(t)$ have speeds bounded above by 1,
by \eqref{ja3.2}-\eqref{ja3.1} and \eqref{n26.16}. Hence, $|D\bZ(t)|\leq n+2$ for all $t$. The sharp upper bound is $3/2$ but we omit the proof because we do not need this sharp bound in our argument.

\medskip

(iii)
If $1\leq j,k\leq n_1$ and $k\neq j$ then $\calY^B_j(t^B_j) < \calY^B_k(t^B_j)$ and  $\calY^C_j(t^C_j) < \calY^C_k(t^C_j)$. It follows that $\calZ_1^B(t) = \calY_j^B(t)$ in a neighborhood of $t_j^B$, and $\calZ_1^C(t) = \calY_j^C$ in a neighborhood of $t^C_j$, for $j=1,\dots, n_1$. 

The following table contains values of some derivatives,  for $j=1,\dots, n_1$.
\begin{table}[h] 
\begin{tabular}{|c|c|c|c|c|c|c|}
\hline
\backslashbox{Function}{Time \Tstrut}
& \Tstrut \Bstrut $t_1^B -$ & $t_1^B +$ & $t_j^B -$ & $t_j^B +$ &$t_j^C -$ & $t_j^C+$  \\
\hline   
 \Tstrut \Bstrut $D  \calZ^B_{0}$ &  $\calV_1^B$ & 0 & $\calV^B_j$ & 0 & 0 & $\calV^B_{j+1}$ \\
\hline 
 \Tstrut \Bstrut $D \calZ^B_1$ & 0 &  $\calV_1^B$ & 0 & $\calV^B_j$ & ? & ?\\
\hline
 \Tstrut \Bstrut $D  \calZ^C_{0}$ &  $\calV_1^B$ &  $\calV_1^C$ & 0 & $\calV_j^C$ & $\calV_j^C$ & 0 \\
\hline
 \Tstrut \Bstrut $D \calZ^C_1$ & 0 & 0 & ? & ? & 0 & $\calV_j^C$\\
\hline
\end{tabular}
\\
{\qquad}\\
\caption{}
\label{tab1}
\end{table}

Question marks in  the table indicate numbers dependent on the initial gaps $\calG^B_j(0)$ and $\calG^C_j(0)$. The explicit expressions for these numbers do not fit in the table. Since these  expressions do not play any   role in our proof, we  omitted them from the table.

\medskip

(iv)
For $j=1,\dots, n_1$,
\begin{align}
\label{d22.01}
D \calZ^B_{0}(t^B_j) &= D  \calZ^B_1(t^B_j-), \\
\label{d22.11}
D \calZ^B_{1}(t^B_j) &= D  \calZ^B_0(t^B_j-), \\
\label{d22.02}
D \calZ^C_{0}(t^B_j) &= D  \calZ^C_{0}(t^B_j-) + \frac12 D  \calZ^B_{0}(t^B_j-) - \frac12 D  \calZ^B_{1}(t^B_j-), \\
\label{d22.03}
D \calZ^C_{0}(t^C_j) &= D  \calZ^C_1(t^C_j-), \\
\label{d22.13}
D \calZ^C_{1}(t^C_j) &= D  \calZ^C_0(t^C_j-), \\
\label{d22.04}
D \calZ^B_{0}(t^C_j) &= D  \calZ^B_{0}(t^C_j-) + \frac12 D  \calZ^C_{0}(t^C_j-) - \frac12 D  \calZ^C_{1}(t^C_j-).
\end{align}

For $k=2,\ldots,m$, the function $D\calZ^A_k$  changes only at $t^A_k$ (jumps from $0$ to $1$), and at $t^A_{k-1}$ (jumps from $1$ to $0$). Thus, for $k=1,\ldots,m-1$,
\begin{equation}
\label{j11.1b}
\begin{aligned}
& D \calZ^A_{k}(t^A_k) = D  \calZ^A_{k+1}(t^A_k-)=1,\\
& D \calZ^A_{k+1}(t^A_{k}) = D  \calZ^A_{k}(t^A_{k}-)=0.
\end{aligned}
\end{equation}

\medskip

(v)
The functions $\calZ^B_0$ and $\calZ^C_0$ can be understood as projections onto $L_1$ and $L_2$, respectively, of the vector function $\calZ^A_0(t):=\frac{2}{\sqrt{3}}\rpar{ \calZ^B_{0}(t) u_2 + \calZ^C_{0}(t) u_1 }$. We have not included $\calZ^A_0$ in the vector $\bZ$ because this function is ``represented'' in $\bZ$ via $\calZ^B_{0}$ and $ \calZ^C_{0}$.

\medskip

(vi)
Discontinuities of $D \bZ$ occur for two reasons: (a) functions $\calY$ exchange their order, and thus, two adjacent functions $\calZ^A_k$ and $\calZ^A_{k+1}$, or $\calZ^B_k$ and $\calZ^B_{k+1}$, or $\calZ^C_k$ and $\calZ^C_{k+1}$ exchange velocities, or (b) the discontinuity occurs at time $t_k^B$ or $t_k^C$. In the latter case, we can use Table \ref{tab1}  to check that the vector functions $\calZ^A_0$, $w_1\calZ_1^B$, and $w_2\calZ^C_1$, representing three equal point masses, satisfy conservation laws for momentum (total velocity) and kinetic energy at each of these times. 

\medskip
 
(vii)
All points $\calY^B_j$ that are moving at time $t^B_k$, do so with a speed larger than that of $\calY^B_k$, i.e., if $D\calY^B_j(t^B_k-)>0$ then $D\calY^B_j(t^B_k-)=D\calY^B_j(t^B_k)>D\calY^B_k(t^B_k)$. It follows that if $j>k$ then  $\calY^B_j$ can meet  $\calY^B_k$ only at a $t>t^B_k$, and so $D \calY^B_j(t-)=0$. This implies that no triplet of functions in the family $\calZ^B_j$, $j\geq 1$ can meet at the same time. For similar reasons,    $\calZ^B_0(t^B_j)=\calZ^B_1(t^B_j)$, but $\calZ^B_2(t^B_j)>\calZ^B_1(t^B_j)$. 
An analogous remark applies to functions in the family $\calY^C_j$, $j\geq 1$, and to $\calZ^C_0,\calZ^C_1$ and $\calZ^C_2$. 

For times $t\leq t^A_1$,  exactly one of the points $\calZ^A_k$, $k=1,\ldots,m$, is moving and all the other components of $\bZ$ are still. It follows that no triplet of functions $\calZ^A_k$, $k=1,\ldots,m$, can be at the same point for $t\geq 0$, since after $t^A_1$ all of them but $\calZ^A_1$ are constant.

We conclude that every discontinuity of $D \bZ$ occurs only at a ``collision time'' of one of the following pairs of functions: 
(a)  $\calZ^A_k$ and $\calZ^A_{k-1}$ for $k=2,\ldots,m$; 
(b) either $\calZ^{B}_{0}$ or $\calZ^{C}_{0}$ with either $\calZ^{B}_{1}$ or $\calZ^{C}_{1}$; 
(c)  $\calZ^{B}_k$ and $\calZ^{B}_{k+1}$ for some $1\leq k < n_1$; or 
(d)  $\calZ^{C}_k$ and $\calZ^{C}_{k+1}$ for some $1\leq k < n_1$. 
Thus no simultaneous collisions occur.

\end{remark}

In the next proposition, we will consider a sequence $\kpar{\bS^k}$ of function families, as in \eqref{f10.1}, and corresponding positive numbers $\kpar{\eps_k}$. We will add the index $k$ to the corresponding functions, as in $a_j^k, b_j^k, c_j^k$ and $\calX^{A,\eps_k}_j,\calX^{B,\eps_k}_j,\calX^{C,\eps_k}_j$.

\begin{proposition}
\label{pr:s1}
Fix $n_1,m,\rho >0$, and let $\kpar{\eps_k}$ be a sequence of positive numbers converging to zero. Let $\kpar{\bS^k}$ be a sequence of families of $2n_1+m$ functions, as in \eqref{f10.1}, such that $\bS^k$ satisfies $(\eps_k,\rho)$ initial conditions at $t_0=0$ (see Definition \ref{de:ic}). Suppose  that $|D \bS^k(0)|=1$ for all $k$, and that $\lim_{k\to \infty} D  a_m^k(0) =w_0 = (0,1)$. Let $\bX^{\eps_k}$ be the associated vector of scaled point masses as in \eqref{eq:bX}, and  assume that $\bX^{\eps_k}(0)$ converge to  $\bX_0$ as $k\to\infty$. Let $\bZ$ be as in \eqref{eq:nz}, with $\bZ(0) = \bX_0$. 
Then $\bZ(t)$ is well defined for all $t\geq 0$, and
  $D {\bX}^{\eps_k}$ converge to $D \bZ$  in the Skorohod space $\calD[0,T]$ for every $T>0$, as $k\to \infty$.
\end{proposition}

\begin{proof}
It is straightforward to check that $(\eps_k,\rho)$ initial conditions for $\bS^k$ imply that $\bX_0$ satisfies conditions \eqref{d2.1} and \eqref{n26.10}. To see \eqref{j5.01}, note that convergence of $\bX^{\eps_k}(0)$ implies that
\begin{align}\label{f12.2}
\frac{w_1\cdot(b^{k}_1(0) - a^{k}_1(0))-2}{w_2\cdot(c^{k}_1(0) - a^{k}_1(0))-2} = \frac{\calX^{B,\eps_k}_1(0)-\calX^{B,\eps_k}_0(0)}{\calX^{C,\eps_k}_1(0)-\calX^{C,\eps_k}_0(0)}
\end{align}
converges as $k\to 0$. By \eqref{eq:ic_ratio}, each of these fractions is either smaller than or equal to $2/3$, or larger than or equal to $3/2$. Convergence implies that the first of the inequalities holds for all large enough $k$, or the second does. Without loss of generality (by switching labels $B$ and $C$ if necessary) we can assume that, eventually, the fractions in \eqref{f12.2} are less than or equal to $2/3$, from which \eqref{j5.01} follows. 
Recall from Remark \ref{re:vz} (i) that $\bZ(t)$ is then well defined for all $t\geq 0$.

Let $T^*$ be the supremum of all $T>0$ such that $D  \bX^{\eps_k}$ converge to $D  \bZ$ in the Skorohod space $\calD[0,T]$. Assume that $T^*<\infty$. We divide the remaining part of the proof into three steps. In the first one, we show that $T^* >0$. In the second step, we show that some component of   $D \bZ$ must be discontinuous at $T^*$. In the third step we show that no component of $D \bZ$ can be discontinuous at time $T^*$.

\begin{step}
In Remark \ref{re:vz} (vii), we argued that at each discontinuity of $D \bZ$, exactly two functions of type $\calZ$ meet. Since the first time this occurs is at $t^A_{m-1} = \calG^A_{m}(0)\geq \rho^{-1}$, it follows that $D \bZ$ is constant in $[0,\rho^{-1})$.  On the other hand, discontinuities of $D \bX^{\eps_k}$ occur only at collisions between discs, and because $(\eps_k,\rho)$ initial conditions hold, each pair of discs are initially at a distance no smaller than $\eps_k\rho^{-1}$. Each disc has speed bounded above by $1$, hence it takes at least $\eps_k\rho^{-1}/2$ units of time for two discs to collide. This shows that $D \bX^{\eps_k}(t) = D \bX^{\eps_k}(0)$  for  $t\in [0,\rho^{-1}/2)$. Our assumption on the initial velocities and equations \eqref{eq:dA}-\eqref{eq:dC} show that $D \bX^{\eps_k}(0)$ converge to $(1,0,\ldots,0)$. According to definitions in Section \ref{de:z}, this is the same as $D \bZ(0)$. It follows that $D \bX^{\eps_k} \to D  \bZ$ uniformly in $[0,\rho^{-1}/2)$, which implies that $T^*\geq\rho^{-1}/2>0$.
 \end{step}

\begin{step}
Let $1\leq j\leq n_1$ be fixed. If $D \calZ^B_j$ is continuous at $T^*$, then for some $\alpha>0$, we have $\calZ^B_{j+1}(T^*) -4\alpha > \calZ^B_{j}(T^*)>\calZ^B_{j-1}(T^*)+4\alpha$, and also $D \calZ^B_j(t)$ is constant in $[T^*-\alpha, T^*+\alpha]$.  By Lemma \ref{le:sk_unif},  for large enough $k$, we have $\calX^{B,\eps_k}_{j+1}(T^*) -3\alpha > \calX^{B,\eps_k}_{j}(T^*)>\calX^{B,\eps_k}_{j-1}(T^*)+3\alpha$. Since $|D \calX^{B,\eps_k}_i|$ is bounded by 1 for all $k$ and $i$, for all $t\in[T^*- \alpha,T^* + \alpha]$ and large $k$, we have 
$\calX^{B,\eps_k}_{j+1}(t) -\alpha > \calX^{B,\eps_k}_{j}(t)>\calX^{B,\eps_k}_{j-1}(t)+\alpha$, which implies that $B_j$ can hit neither $B_{j+1}$ nor $B_{j-1}$ in $[\eps_k(T^*-\alpha),\eps_k(T^*+\alpha)]$. 
The discs satisfy the initial conditions $(\eps_k,\rho)$ and $\eps_k\to0$ so for large $k$, the only discs whose centers are within distance $2\sqrt{2}$ from the center of $B_j$ are $B_{j+1}$ and $B_{j-1}$.
The speeds of the discs are bounded by 1 so, for large $k$ and  $t\in[0, \eps_k(T^*+\alpha)]$, the center of $B_j$ and the center of every other disc, except $B_{j+1}$ and $B_{j-1}$, are at a distance exceeding $2\sqrt{2} -0.1$.
Thus 
 $B_j$ does not take part in any collision in $[\eps_k(T^*-\alpha),\eps_k(T^*+\alpha)]$, and so $D \calX^{B,\eps_k}_j$ is constant in this interval for all large enough $k$.   Corollary \ref{co:extension} implies that $D \calX^{B, \eps_k}_j$ converge to $D \calZ^B_j$ in $\calD[0,T^*+\alpha]$.
 The same argument applies if $D \calZ^A_j$, with $j\geq 2$, or $D \calZ^C_j$, with $j\geq 1$, are continuous at $T^*$. Under these assumptions, there is $\alpha >0$ such that $D \calX^{A,\eps_k}_j $ converge to $ D \calZ^A_j$, and  $D  \calX^{C,\eps_k}_j $ converge to $D  \calZ^C_j$ in $\calD[0,T^*+\alpha]$, respectively.

The continuity times of $D \calZ^B_0$, $D \calZ^C_0$, and $D \calZ^A_1$ are the same. If these functions are continuous at $T^*$ then there must exist $\alpha >0$ such that $\calZ^A_1(T^*) >  \calZ^A_2(T^*) +4\alpha$, $\calZ^B_1(T^*) >  \calZ^B_0(T^*) +4\alpha$, and $\calZ^C_1(T^*) >  \calZ^C_0(T^*) +4\alpha$. The argument continues as in the previous paragraph so that we can conclude that for large enough $k$, the functions $D \calX^{A,\eps_k}_1$, $D \calX^{B,\eps_k}_0$, and $D \calX^{C,\eps_k}_0$ are constant in $[T^*- \alpha,T^* + \alpha]$, and they converge respectively to $D \calZ^{A}_1$, $D \calZ^{B}_0$, and $D \calZ^{C}_0$ in $\calD[0,T^*+\alpha]$.

If $T^*$ is a continuity point of $D \bZ$, then the above discussion  applies to all components of $\bZ$. It follows that $D \bX^{\eps_k}$ converge to $D \bZ$ in $\calD[0,T^*+\alpha]$, for some $\alpha >0$, a contradiction with the definition of $T^*$.

\end{step}

\begin{step}
Assume that $D \calZ^B_j$ is discontinuous at $T^*$ for some $0\leq j\leq n_1$. By Remark \ref{re:vz} (vii), we must have $\calZ^B_j(T^*)=\calZ^B_{i}(T^*)$, for either $i=j-1$ or $i=j+1$. Without loss of generality, assume that $i=j+1$; otherwise, replace  $j$ with $j-1$. Note that $D \calZ^B_{j+1}$ must be discontinuous at $T^*$. Also, if $j=0$, $D \calZ^C_0$ must be discontinuous at $T^*$.

By Lemma \ref{le:zero}, there exist
 $\delta_0>0$ and a unique $T_k\in [T^*-\delta_0,T^*+\delta_0]$ such that $D \calX^{B,\eps_k}_j$ and $D \calX^{B,\eps_k}_{j+1}$ are discontinuous at $T_k$. The time $T_k$ is the same for both functions because of conservation of momentum at collision times.
We make $\delta_0>0$ smaller, if necessary, so that $D\bZ$ is constant in $[T^*-2\delta_0, T^*)$ and $(T^*, T^*+2\delta_0]$.

Consider the case
 $j=0$. Then $T^*=t^B_i$ for some $i\geq 1$. Subtracting equations \eqref{d22.01} and \eqref{eq:velpm01} (with $j=0$ and $t=T_k$) we get
\begin{align*}
D  \calZ^ B_0(T^*) - D  \calX^{B,\eps_k}_0(T_k)  &= D  \calZ^B_{1}(T^*-) - D  \calX^{B,\eps_k}_1(T_k-) + O(\eps_k(1+T^*)).
\end{align*}
We obtain  $D \calX^{B,\eps_k}_0\to D \calZ^B_0$ in $\calD[0,T^*+\delta_0]$ by
applying Lemma \ref{le:extension} 
with $\delta = \delta_0$, $T=T^*$, $t_k = T_k$, $f= \calZ^ B_0$, $f_k = \calX^{B,\eps_k}_0$, $g= \calZ^B_{1}$, $g_k = \calX^{B,\eps_k}_1$ and $\alpha_k = c \eps_k(1+T^*)$.

Subtracting equations \eqref{d22.11} and \eqref{eq:velpm03} (with $j=0$ and $t=T_k$) we get
\begin{align*}
D  \calZ^ B_1(T^*) - D  \calX^{B,\eps_k}_1(T_k)  &= D  \calZ^B_{0}(T^*-) - D  \calX^{B,\eps_k}_0(T_k-) + O(\eps_k(1+T^*)).
\end{align*}
We obtain  $D \calX^{B,\eps_k}_1\to D \calZ^B_1$ in $\calD[0,T^*+\delta_0]$ by
applying Lemma \ref{le:extension} 
with $\delta = \delta_0$, $T=T^*$, $t_k = T_k$, $f= \calZ^ B_1$, $f_k = \calX^{B,\eps_k}_1$, $g= \calZ^B_0$, $g_k = \calX^{B,\eps_k}_0$ and $\alpha_k = c \eps_k(1+T^*)$.

Subtracting equations \eqref{d22.02} and \eqref{eq:velpm05} at $t=T_k$, we get
\begin{align*}
D  \calZ^C_0 (T^*) - D  \calX^{C,\eps_k}_0(T_k) 
&= 
O(\eps_k(1+T^*)) + D  \calZ^C_{0}(T^*-) - D \calX^{C,\eps_k}_0(T_k-) + \\
&\qquad + \frac12 D  \calZ^B_{0}(T^*-)  -\frac12 D \calX^{B,\eps_k}_0(T_k-) + \\
&\qquad  - \frac12 D  \calZ^B_1(T^*-) +\frac12 D \calX^{B,\eps_k}_1(T_k-).
\end{align*}
We obtain $D \calX^{C,\eps_k}_0\to D \calZ^C_0$ in $\calD[0,T^*+\delta_0]$ by
applying Lemma \ref{le:extension} 
with $\delta = \delta_0$, $T=T^*$, $t_k = T_k$, $f= \calZ^C_0$, $f_k = \calX^{C,\eps_k}_0$, $g= (\calZ^B_0-\calZ^B_1)/2$, $g_k = (\calX^{B,\eps_k}_0-\calX^{B,\eps_k}_1)/2$ and $\alpha_k = c \eps_k(1+T^*)$.

If $j\geq 1$, then the discontinuities of $D \calZ^B_j$ and $D \calZ^B_{j+1}$ occur because the functions of type $\calY$ associated to $\calZ^B_j$ and $\calZ^B_{j+1}$ exchange order; thus, $D  \calZ^B_j(T^*) = D  \calZ^B_{j+1}(T^*-)$ and $D  \calZ^B_{j+1}(T^*) = D  \calZ^B_j(T^*-)$. Subtracting these equations respectively from \eqref{eq:velpm01} and \eqref{eq:velpm03} we obtain
\begin{align*}
D  \calZ^ B_j(T^*) - D  \calX^{B,\eps_k}_j(T_k)  &= D  \calZ^B_{j+1}(T^*-) - D  \calX^{B,\eps_k}_{j+1}(T_k-) + O(\eps_k(1+T^*)), \\
D  \calZ^ B_{j+1}(T^*) - D  \calX^{B,\eps_k}_{j+1}(T_k)  &= D  \calZ^B_{j}(T^*-) - D  \calX^{B,\eps_k}_{j}(T_k-) + O(\eps_k(1+T^*)).
\end{align*}
\end{step}\noindent Applying Lemma \ref{le:extension} once again, we obtain that $D \calX^{B,\eps_k}_j\to D \calZ^B_j$ and $D \calX^{B,\eps_k}_{j+1}\to D \calZ^B_{j+1}$ in $\calD[0,T^*+\alpha]$ for some $\alpha>0$.

A completely analogous argument applies in the cases when
 $D  \calZ^C_j$ is discontinuous at $T^*$, or $D  \calZ^A_j$ is discontinuous at $T^*$ for some $j\geq 2$. Finally, the argument in
the case when $D  \calZ^A_1$ is discontinuous at $T^*$ can be based
on the observation that
$D  \calZ^A_1 = D  \calZ^B_0 + D  \calZ^C_0$ and  $\calX^{A,\eps_k}_1 = \calX^{B,\eps_k}_0 + \calX^{C,\eps_k}_0$.

For all $F=A,B,C$ and all $j$, we obtain $D \calX^{F,\eps_k}_{j}\to D \calZ^F_{j}$ in $\calD[0,T^*+\alpha]$ for some $\alpha>0$,
which contradicts the definition of $T^*$, unless $T^*=\infty$.
\end{proof}

The last discontinuity of $D \calZ^A_1$ is at time $t^C_{n_1}$. We have by \eqref{n26.12}-\eqref{n26.15},
\begin{align}
\label{eq:tc}
t^C_{n_1} &= \sum_{k=2}^m \calG^A_k(0) + \sum_{k=1}^{n_1} 
\left(\frac{2^{2k-1}}{3} \calG^B_k(0) + \frac{2^{2k}}{3} \calG^C_{k}(0)\right).
\end{align}
In view of \eqref{d2.1} and \eqref{n26.10}, it follows that 
\begin{align}
\label{eq:tc_bds}
\frac1{\rho}\rpar{m-1 +  \frac{2^{2n_1+1}-2}{3}} < 
t^C_{n_1}  
< m-1+\frac{2^{2n_1+1}}{3}.
\end{align}

Recall that given $n$, the total number of discs, we define $n_1$ and $n_2$ by $n_1 = \lfloor n/3 \rfloor$, $n_2 = n - 2 n_1$.

\begin{proposition}
\label{th:epsT}
Let $T=n_2+2^{2n_1}$ and fix $\rho\geq 2$. There is $\eps_* =\eps_*(\rho)>0$ such that for any $\eps  \in (0,\eps_*)$,  any $\lambda>0$, and any $1\leq m\leq n_2$,  the following holds: If $\bS$ is a family of $2n_1+m$ functions satisfying $(\eps,\rho)$ initial conditions at time $t_0$, $|D \bS(t_0)|=\lambda$ and $|D  a_m(t_0)-\lambda w_0|\leq\lambda\eps$, then $\bS$ undergoes at least $m-1+n_1(n_1+1)$ collisions in $[t_0,t_0+\eps \lambda^{-1} T)$, and $\bS$ satisfies $(\eps(1+2T),\rho(1+3T))$ initial conditions at time $t_0+\eps \lambda^{-1}T$.
\end{proposition}

\begin{proof}
By a time translation, we can assume that $t_0=0$. 

We will show that it is sufficient to prove the proposition for $\lambda=1$.
Assume that the proposition is true for $\lambda=1$. Let $\bS$ be as in the statement of the proposition and let $\widetilde\bS(t) = \bS(\lambda^{-1}t)$. It is straightforward to check that $\widetilde\bS$ satisfies the conditions in the proposition with $\lambda=1$, and thus $\widetilde\bS$ undergoes at least $m-1+n_1(n_1+1)$ collisions in $[0,\eps T)$, and satisfies $(\eps(1+2T),\rho(1+3T))$ initial conditions at time $\eps T$. Let $\widetilde\bX$ and $\bX$ be the vectors of point masses associated to $\widetilde\bS$ and $\bS$, resp., as in \eqref{eq:bX}. It is easy to check that $\widetilde\bX(t) = \bX(\lambda^{-1}t)$. This shows that the proposition holds for any $\lambda>0$  by time scaling. It is evident from this argument that the value of $\eps_*$ which works for $\lambda=1$ also works for any $\lambda>0$, and so $\eps_*$ is independent of $\lambda$.

Assume that the proposition is false for $\lambda=1$. Then there is a sequence $\eps_k\to0$, a sequence of families $\bS^k$ of $2n_1+m$ functions satisfying $(\eps_k,\rho)$ initial conditions at time zero, with  $|D \bS^k(0)|=1$ and $|D a_m(0)-w_0|\leq\eps_k$, such that $\bS^k$ undergoes fewer than $m-1+n_1(n_1+1)$ collisions by time $\eps_kT$, or $(\eps_k(1+2T),\rho(1+3T))$  conditions at time $\eps_k T$ do not hold.

Let $\bX^k$ be the vector of point masses associated to $\bS^k$. It follows from Definition \ref{de:ic} (iii) that $| \bX^{k}(0)|$ is uniformly bounded. Hence, passing to a subsequence if necessary, we can assume that $ \bX^{k}(0)$ converge to some $\bX_0$. By Proposition \ref{pr:s1}, if $\bZ$ is as in \eqref{eq:nz} with $\bZ(0)=\bX_0$ then $D \bX^k$ converge to $D \bZ$ in $\calD[0,T]$. 

The function $D \bZ$ has a discontinuity at each of the times in \eqref{n26.12} (except for $t^A_m=0$), which accounts for $m-1 + 2n_1$ discontinuities. 

We will argue that functions $\calY^B_j$ reverse their order by time $t^C_{n_1}$. It follows from \eqref{f14.1} and \eqref{eq:yb} that at time 0, their order is $\calY^B_1(0) < \calY^B_2(0) < \dots < \calY^B_{n_1}(0)$.
Consider $j\geq 2$.
At time $t^B_j$, the derivative of $\calY^B_j$ becomes $ \calV^{B}_{j} = 3/2^{2j-1}$. At this time, the distance between $\calY^B_j$ and $\calY^B_{j+1}$ is $\calG^B_{j+1}$. Hence,  $\calY^B_j$ will take the value  $\calY^B_{j+1}(t^B_j)$ at time $t_*:=t^B_j + (2^{2j-1}/3)\calG^B_{j+1}$. 
By \eqref{n26.12} and \eqref{n26.14},
\begin{align*}
t^B_{j+1} = t^B_{j+1} - t^B_{j} + t^B_{j}\geq
t^B_{j+1} - t^C_{j}+ t^B_{j} &= \frac{2^{2j+1}}{3} \calG^{B}_{j+1}(0)+ t^B_{j}  > t_*.
\end{align*}
This means that $\calY^B_j$ will cross  $\calY^B_{j+1}$ before the latter starts moving. After $\calY^B_{j+1}$ starts moving, $\calY^B_j$ will have a greater derivative than that of $\calY^B_{j+1}$. This implies that for every $i>j+1$, $\calY^B_j$ will cross  $\calY^B_{i}$ before $\calY^B_{j+1}$ does.
The proof for $j=1$ is similar. All of these observations imply that the order of the functions $\calY^B$ will be reversed by time $t^C_{n_1}$, which is  possible only if $n_1(n_1-1)/2$ crossings occur between these functions.
A similar reasoning applies to functions of type $\calY^C$.
We have accounted for a total of $m-1+2n_1 + n_1(n_1-1)/2 + n_1(n_1-1)/2 = m-1 + n_1(n_1+1)$ discontinuities.

 Lemma \ref{le:zero} shows that every discontinuity of $D \bZ$ is associated to a unique discontinuity of $D \bX^k$ for $k$ large enough, and thus, $\bS^k$ undergoes at least $m-1 + n_1(n_1+1)$ collisions by time $\eps_k T$, since $T>t^C_{n_1}$ by \eqref{eq:tc_bds}.

It remains to show that $(\eps_k(1+2T),\rho(1+3T))$ initial conditions hold at time $\eps_k T$ for large $k$. We start with some estimates for $\bZ(T)$. We have already pointed out that $\calY^B_{j}(t^C_{n_1}) \geq \calY^B_{j+1}(t^C_{n_1})$ for $j=1,\ldots,n_1-1$. For $j=1,\ldots,n_1$, 
\begin{align}
\label{eq:ybtc}
\calZ^B_{j}(t) = \calY^B_{n_1-j+1}(t) = \calZ^B_{j}(t^B_{n_1}) + \calV^B_{n_1-j+1}\cdot (t-t^B_{n_1})\qquad\text{for } t\geq t^B_{n_1}.
\end{align}
Similarly, for $j=1,\ldots,n_1$, 
\begin{align}
\label{eq:yctc}
\calZ^C_{j}(t) = \calY^C_{n_1-j+1}(t) = \calZ^C_{j}(t^C_{n_1}) + \calV^C_{n_1-j+1} \cdot (t-t^C_{n_1})\qquad\text{for } t\geq t^C_{n_1}.
\end{align}
These formulas and \eqref{n26.16} imply that for $t\geq t^C_{n_1}$ and $j=1,\ldots,n_1$,  functions $\calZ^B_j - \calZ^B_{j-1}$ and $\calZ^C_j - \calZ^C_{j-1}$  have derivatives larger than $D (\calZ^C_1-\calZ^C_0)(t^C_{n_1})=3/2^{2n_1}$. Since $D  \calZ^A_j(t)=0$ for $t\geq t^A_1$ and $j\geq 2$, we conclude that $D \bZ$ is constant for $t>t^C_{n_1}$. This and the fact that $T-t^C_{n_1}\geq 2^{2n_1}/3$ for $m\leq n_2$ (see  \eqref{eq:tc_bds}), allow us to obtain the following bound for $j=1,\ldots,n_1$,
\begin{equation}
\label{j23.1}
\begin{aligned}
\rpar{\calZ^B_j(T) - \calZ^B_{j-1}(T)} \wedge \rpar{\calZ^C_j(T) - \calZ^C_{j-1}(T)} &\geq \frac{3}{2^{2n_1}} ( T - t^C_{n_1}) \geq 1 .
\end{aligned}
\end{equation}

It follows from \eqref{eq:ya}-\eqref{j11.1a} that 
$\calZ^A_2(T) = \calZ^A_2(t^A_1)= \calZ^A_1(0)$.
Recall from \eqref{j11.4} that $\calZ^A_1(t) = \calZ^B_0(t) + \calZ^C_0(t)$ for $ t\geq 0$.
It is easy to see that $\calZ^B_0(t) \geq \calZ^B_0(s)$ and $\calZ^C_0(t)\geq \calZ^C_0(s)$ for $t\geq s$. 
We have 
$\calZ^B_{0}(t)= \calZ^B_1(t^B_{n_1}) + 3 (1/2)^{2n_1+1} \cdot  (t-t^C_{n_1}) $  for  $t\in[t^C_{n_1}, \infty)$, by \eqref{ja8.6}. These remarks and \eqref{eq:tc_bds} imply that
\begin{equation}
\label{j23.2}
\begin{aligned}
\calZ^A_1(T) - \calZ^A_2(T) &=  \calZ^A_1(T) - \calZ^A_1(0) = \calZ^B_0(T) - \calZ^B_0(0)+ \calZ^C_0(T) - \calZ^C_0(0) \\
&\geq 
\calZ^B_0(T) - \calZ^B_0(t^C_{n_1})
=
\frac{3}{2^{2n_1+1}}(T - t^C_{n_1}) \geq \frac{1}{2}.
\end{aligned}
\end{equation}

By the  definition \eqref{eq:ya}-\eqref{j11.1a} of $\calZ^A_j(t)$ and \eqref{d2.1},
\begin{align}
\label{j23.3}
\calZ^A_j(T) - \calZ^A_{j+1}(T) &= \calZ^A_{j-1}(0) - \calZ^A_j(0) \geq \frac{1}{\rho}\qquad \text{ for } k=2,\ldots,m-1.
\end{align}

We recall from the paragraph preceding
\eqref{j23.2} that
 $\calZ^A_1(t) = \calZ^B_0(t)  + \calZ^C_0(t)$
and 
$\calZ^B_{0}(t)= \calZ^B_1(t^B_{n_1}) + 3 (1/2)^{2n_1+1} \cdot  (t-t^C_{n_1}) $  for  $t\in[t^C_{n_1}, \infty)$. We argued earlier in the proof that the order of functions $\calY^B$ is reversed over the interval $[0, t^B_{n_1}]$. Hence, $\calZ^B_1(t^B_{n_1}) = \calZ^B_{n_1}(0)$, $\calZ^C_0(T) = \calZ^C_{n_1}(0)$, and, therefore,  $\calZ^B_0(T) = \calZ^B_{n_1}(0) + (T-t^C_{n_1}) 3/2^{2n_1+1}$. 
 From \eqref{n26.10}, we have $\calZ^A_{1}(0) \leq 1$, $\calZ^B_{n_1}(0) \leq \calZ^B_0(0)+n_1$, and $\calZ^C_{n_1}(0) \leq \calZ^C_0(0)+n_1$. Putting all of this together, with \eqref{eq:tc_bds}, yields
\begin{align}\label{f15.1}
\calZ^A_{1}(T) &\leq 
\calZ^A_{1}(T)-\calZ^A_{1}(0)+\calZ^A_{1}(0)
\leq 
 \calZ^B_0(T) - \calZ^B_0(0)+ \calZ^C_0(T) - \calZ^C_0(0) +1\\
&\leq
\calZ^B_{n_1}(0) + (T-t^C_{n_1}) 3/2^{2n_1+1} 
 - \calZ^B_0(0)+ \calZ^C_{n_1}(0) - \calZ^C_0(0) +1\notag\\
&\leq
\calZ^B_0(0)+n_1 + (T-t^C_{n_1}) 3/2^{2n_1+1} 
 - \calZ^B_0(0)+ \calZ^C_0(0)+n_1 - \calZ^C_0(0) +1\notag\\
&\leq
1 + 2n_1 + 3T/2^{2n_1+1} \leq 1+T.\notag
\end{align}

Fix $\eta\in (0, T/(1+3T))$. By Lemma \ref{le:sk_unif}, $\bX^k\to\bZ$ uniformly in $[0,T]$. This and \eqref{f15.1} imply that for large $k$, 
\begin{align}
w_1\cdot a_1(\eps_k T) &\geq (1-\eta)\eps_k \calZ^B_0(T) = (1-\eta)\eps_k \calZ^B_{n_1}(0) \geq (1-\eta)\eps_k\rho^{-1} >0, \label{f15.2}\\
w_2\cdot a_1(\eps_k T) &\geq (1-\eta)\eps_k \calZ^C_0(T) = (1-\eta)\eps_k \calZ^C_{n_1}(0) \geq (1-\eta)\eps_k\rho^{-1} >0, \label{f15.3}\\
w_0\cdot a_1(\eps_k T) &\leq (1+\eta)\eps_k \calZ^A_1(T) \leq \eps_k (1+\eta)(1+T) \leq \eps_k(1+2T).\label{f15.4}
\end{align}

We will refer to ``conditions (i)-(iv)'' below. These are conditions (i)-(iv)  in Definition \ref{de:ic}  with $t_0=T$, $\eps = \eps_k(1+2T)$, and $\rho$ replaced with $\rho (1+3T)$.

Condition (i) is satisfied due to \eqref{f15.2}-\eqref{f15.4}.

Since discs' speeds (derivatives of functions in $\bS^k$) are bounded from above by 1 and $\bS^k$ satisfies $(\eps_k,\rho)$ initial conditions, we have  $|P_{u_0}a_j(\eps_k T)|\leq |P_{u_0}a_j(0)| + \eps_kT \leq \eps_k(1+T)$.
This and similar estimates for $b_j$'s and $c_j$'s show that condition (ii) holds.

Recall Definitions \ref{de:gaps} and \ref{de:ic}. Since $\bS^k$ satisfies $(\eps_k,\rho)$ initial conditions and derivatives of components of $\bS^k$ are bounded  by 1, we have for every $k$ and $j=1,\ldots,m-1$,
\begin{align*}
w_0\cdot \rpar{a_j(\eps_k T) - a_{j+1}(\eps_k T)}-2 &\leq w_0\cdot\rpar{a_j(0) - a_{j+1}(0)}-2 + 2\eps_k T\leq \eps_k(1+2T),
\end{align*}
from which it follows that 
$\max_{2\leq j\leq m}\Gamma^{A}_j(\eps_k T)\leq \eps_k(1+2T)$. 
Similarly,  $\max_{1\leq j\leq n_1}\Gamma_j^{B}(\eps_k T)$ and $\max_{1\leq j\leq n_1}\Gamma_j^{C}(\eps_k T)$ are bounded  by $ \eps_k(1+2T)$. It follows that $\rho^+_\Gamma(\eps_k T)\leq \eps_k(1+2T)$, i.e., the upper bound in condition (iii) holds. 

For the lower bound, given $\eta$ as above, and $k$ large enough, we obtain from \eqref{j23.1} that
\begin{align*}
\min_{1\leq j\leq n_1} \Gamma_j^{B}(\eps_k T) \wedge\min_{1\leq j\leq n_1} \Gamma_j^{C}(\eps_k T) \geq (1-\eta)  {\eps_k}.
\end{align*}
 By \eqref{j23.2} and \eqref{j23.3},  for $k$ large enough, we  have $\min_{2\leq j\leq m}\Gamma_j^{A}(\eps_k T) \geq (1-\eta) \eps_k/\rho$, since $\rho\geq 2$. By our choice of $\eta$, we have  $1-\eta\geq (1+2T)/(1+3T)$. Hence, we obtain the lower bound in condition (iii), i.e.,
\begin{align*}
\rho^-_\Gamma(\eps_k T) \geq \frac{\eps_k (1+2T)}{\rho (1+3T)}.
\end{align*}

It remains to verify condition (iv). By Lemma \ref{le:sk_unif}, for large $k$, 
\begin{align}
\label{j5.03}
\frac{{\calX^{B,\eps_k}_1(T) - \calX^{B,\eps_k}_{0}(T)}}{{\calX^{C,\eps_k}_1(T) - \calX^{C,\eps_k}_{0}(T)}} + \eta
\geq 
\frac{{\calZ^{B}_1(T) - \calZ^{B}_{0}(T)}}{{\calZ^{C}_1(T) - \calZ^{C}_{0}(T)}}.
\end{align}
We use \eqref{ja8.6}, \eqref{ja8.11} and \eqref{d2.1} in the following computation,
\begin{align*}
\frac{{\calZ^{B}_1(T) - \calZ^{B}_{0}(T)}}{{\calZ^{C}_1(T) - \calZ^{C}_{0}(T)}} &= 
\frac{{\calZ^{B}_1(T) - \calZ^{B}_{1}(t^B_{n_1}) - 3(1/2) ^{2n_1+1} (T-t^C_{n_1}) }}{{\calZ^{C}_1(T) - \calZ^{C}_{1}(t^C_{n_1})}}\\
&=
\frac{{3(1/2) ^{2n_1-1}(T - t^B_{n_1})
 - 3(1/2) ^{2n_1+1} (T-t^C_{n_1}) }}
{{3(1/2) ^{2n_1}(T  -t^C_{n_1})}}\\
&=
\frac{{2(T - t^B_{n_1}) }}
{{T  -t^C_{n_1}}} - \frac12
= \frac32 + \frac{2(t^C_{n_1} - t^B_{n_1})}{T-t^C_{n_1}}\\
&= \frac32 + \frac{2 ( 2^{2n_1}/3 ) \calG^C_{n_1}(0)}{T-t^C_{n_1}} \geq \frac32 + \frac{2^{2n_1+1} }{3T\rho}.
\end{align*}
If we choose $\eta$ smaller than $2^{2n_1+1}/(3T\rho)$ and combine the last estimate with \eqref{j5.03}, we obtain the second inequality in \eqref{eq:ic_ratio} for large $k$. This completes the proof that the family $\bS^k$ satisfies $(\eps_k(1+2T),\rho(1+3T))$ initial conditions at time $\eps_k T$, for large  $k$, which contradicts our choice of the sequence $\bS^k$, and completes the proof.
\end{proof}

\begin{proposition} 
\label{th:prep}
For every $\eps\in (0,1)$ and $\rho>3/2$, there is a family $\bS$ of $2n_1+1$ functions and $T_{0}<0$ such that:
\begin{enumerate}[(i)]
\item no collisions occur in $(-\infty,T_{0}]$, 
\item $n_1(n_1-1)$ collisions occur in $(T_{0},0)$,  
\item $(\eps,\rho)$ initial conditions hold at time zero for the system $\bS$, and
\item $|D \bS(0)|=1$, and $|D  a_1(0) - w_0|\leq \eps$.
\end{enumerate}
\end{proposition}

\begin{proof}
We will define initial conditions at time zero, and by running time backwards, we will show that  time $T_0$ exists. Recall the vectors $w_0,w_1$, and $w_2$ defined in Section \ref{se:notex}. For $j=1,\dots , n_1-1$, set
\begin{align*}
a_1(0) = (0,0),\ b_1(0) = \rpar{2+\frac23\eps}w_1,\ c_1(0) = (2+\eps)w_2 ,\\
b_{j+1}(0) = b_{j}(0) + (2+\eps)w_1,\ c_{j+1}(0) = c_{j}(0) + (2+\eps)w_2.
\end{align*}
It is straightforward to check that $(\eps,\rho)$ initial conditions hold at time zero. Next, set $D  a_1(0) = (1-\eps) w_0$, and for $j=1,\dots , n_1$,
\begin{align*}
D b_{j}(0) = \kappa (j- n_1) w_1,\qquad
D c_{j}(0) = \kappa (j- n_1) w_2,
\end{align*}
where $\kappa>0$ is  chosen so that $|D \bS(0)|=1$. In this way we make sure that (iv) is satisfied. 

Consider the dynamics of $t \mapsto \bS(-t)$ for $t>0$. Disc $A_1$ moves downwards. Discs $B_j$ (resp. $C_j$) move away from the origin along the line $L_1$ (resp. $L_2$). Since the arrow of time is reversed, the velocities of discs $B_j$ (resp. $C_j$) are increasing in $j$ along $w_1$ (resp. $w_2$). The dynamics of the families $\{B_j\}$ and $\{C_j\}$ are similar to that of the example presented at the end of Example \ref{oc18.2}, in terms of the number of collisions. For this reason, the only possible collisions are between $B_j$ and $B_{j+1}$, and between $C_j$ and $C_{j+1}$, and exactly $n_1(n_1-1)$  collisions occur in $(-\infty,0)$. Thus, we can choose $T_0<0$ such that all the collisions occur in $(T_0,0]$. 
\end{proof}

\begin{proof}[Proof of Theorem \ref{oc20.1}] Note that it will suffice to prove the theorem for $d=2$.

Recall the definitions of $T$ and $\eps_*(\rho)$ from Proposition \ref{th:epsT}. Fix any $\rho_0 >2$ and let $\rho_m = \rho_0(1+3T)^m$ for $m\geq 1$. Fix  $\eps_0 >0$ such that, if we define $\eps_m = \eps_0(1+2T)^m$ for  $1\leq m\leq n_2$, then $\eps_m \in (0,1/2)$ and $\eps_m < \eps_*(\rho_m)$ for all $1\leq m\leq n_2$.

 Let $\lambda_1=1$, $T_1=0$,  $\lambda_{m+1} = \lambda_m/\sqrt{\eps_{m+1}}$ and $T_{m+1}=T_{m}+\eps_m\lambda_m^{-1}T$ for $m\geq 1$. 

\setcounter{step}{-1}
\begin{step}
We apply Proposition \ref{th:prep} with
 $\eps=\eps_1$ and $\rho=\rho_1$. According to the proposition, there exist $T_{0}<0$ and a system $\bS_1$ of $2n_1+1$ discs with properties (i)-(iv)  listed in the statement of  Proposition \ref{th:prep}. Note that $|D \bS_1(t)|=\lambda_1=1$ for $t\leq T_1=0$.
\end{step}

\begin{step}
This is an inductive step. 

Suppose that  $1 \leq m\leq n_2-1$.
Consider the following assumptions on a system $\bS_m$ of $2n_1+m$ discs.

\medskip
(A) The family $\bS_m$  satisfies $(\eps_m,\rho_m)$ initial conditions at time $T_m$,  $|D  a_m(T_m) - \lambda_m w_0|<\lambda_m\eps_m$ and $|D \bS_m(T_m)| = \lambda_m$.
\medskip

Note that for $m=1$, the system $\bS_1$ constructed in Step 0 satisfies these assumptions.

We use Proposition \ref{th:epsT} to conclude that  $\bS_m$ undergoes at least $m-1+n_1(n_1+1)$ collisions in the interval $(T_m,T_{m+1})$, and $\bS_m$ satisfies $(\eps_{m+1},\rho_{m+1})$ initial conditions at time $T_{m+1}$. Since energy is conserved, $|D \bS_m(T_{m+1})| = \lambda_m$. 

We will construct a family of discs $\bS_{m+1}$ by adding a disc to $\bS_m$.

We  define the trajectory $a_{m+1}$ of the center of disc $A_{m+1}$ up to time $T_{m+1}$ as the unique linear function satisfying 
\begin{align}
\label{j25.1}
a_{m+1}(T_{m+1}) &= a_m(T_{m+1}) - (2+\eps_{m+1})w_0, \\ 
\label{j25.2}
D  a_{m+1}(t) &= \lambda_{m+1}\sqrt{1-\eps_{m+1}} w_0, \qquad\text{ for } t \leq T_{m+1}.
\end{align}
Since $|D \bS_m(T_{m+1})| = \lambda_m$, we have $|D a_m(t)| \leq \lambda_m $ for $t\leq T_{m+1}$.
We have assumed that $\eps_{m+1}\in (0,1/2)$, so for $t\leq T_{m+1}$,
\begin{align*}
w_0\cdot D   ( a_m - a_{m+1} )(t)   &\leq \lambda_m -\lambda_{m+1}\sqrt{1-\eps_{m+1}}  \\
&=\lambda_{m+1}\sqrt{\eps_{m+1}} - \lambda_{m+1}\sqrt{1-\eps_{m+1}}  \leq 0.
\end{align*}
It follows that, for $t\leq T_{m+1}$,
\begin{align*}
w_0\cdot (a_m(t) - a_{m+1}(t) ) &\geq w_0\cdot (a_m(T_{m+1}) - a_{m+1}(T_{m+1}) ) = 2 + \eps_{m+1}, 
\end{align*}
which shows that $A_{m+1}$ does not collide with any disc in the system $\bS_m$ before time $T_{m+1}$. This allows us to define a system $\bS_{m+1}(t)$ for $t\leq T_{m+1}$ by adding $a_{m+1}$ to $\bS_m$.

Recall that $\bS_m$ satisfies $(\eps_{m+1},\rho_{m+1})$ initial conditions at time $T_{m+1}$.
This and \eqref{j25.1} imply that the system $\bS_{m+1}$ of $2{n_1}+m+1$ discs satisfies $(\eps_{m+1},\rho_{m+1})$ initial conditions at time $T_{m+1}$. Since $\eps_{m+1} < 1/2$, by \eqref{j25.2}, 
\begin{align*}
|D  a_{m+1}(T_{m+1}) - \lambda_{m+1} w_0| = \lambda_{m+1}\rpar{1-\sqrt{1-\eps_{m+1}}} \leq \lambda_{m+1}\eps_{m+1}.
\end{align*}
It is straightforward to check from the definitions that $ | D  \bS_{m+1}( T_{m+1} ) |^2 =  \lambda_{m+1}^2$.

We conclude that $\bS_{m+1}$ satisfies assumptions (A) stated at the beginning of this step with $m $ replaced by $m+1$.

\end{step}

\begin{step}
Consider the system $ \bS_{n_2}$ inductively defined in the previous step. The system undergoes $n_1(n_1-1)$ collisions in $(T_0,T_1)$, 
and at least $m-1+n_1(n_1+1)$ collisions  in $(T_m,T_{m+1})$
for $m=1,\dots,n_2$. 
The sum $n_1(n_1-1)+ \sum_{m=1}^{n_2}(m-1+n_1(n_1+1))$ is equal to the function $f(n)$ defined in \eqref{d31.1}. This completes the proof.
\end{step}
\end{proof}

\section{Small families of balls}\label{smallex}

This section is devoted to examples involving families of $n$ discs with $3\leq n \leq 6$. We start with the example involving only 3 discs.  Recall that an example showing that $K(3,2) \geq 4$ was found by J.D.~Foch and published in \cite{MurCoh}. That example contains the list of initial conditions (positions and velocities of discs) and a schematic drawing of the corresponding trajectories. While we find that drawing very helpful, we do not believe that it is  accurate. For this reason we present our own rendering of the Foch example in Fig.~\ref{fig8}.
\begin{figure} \includegraphics[width=0.9\linewidth]{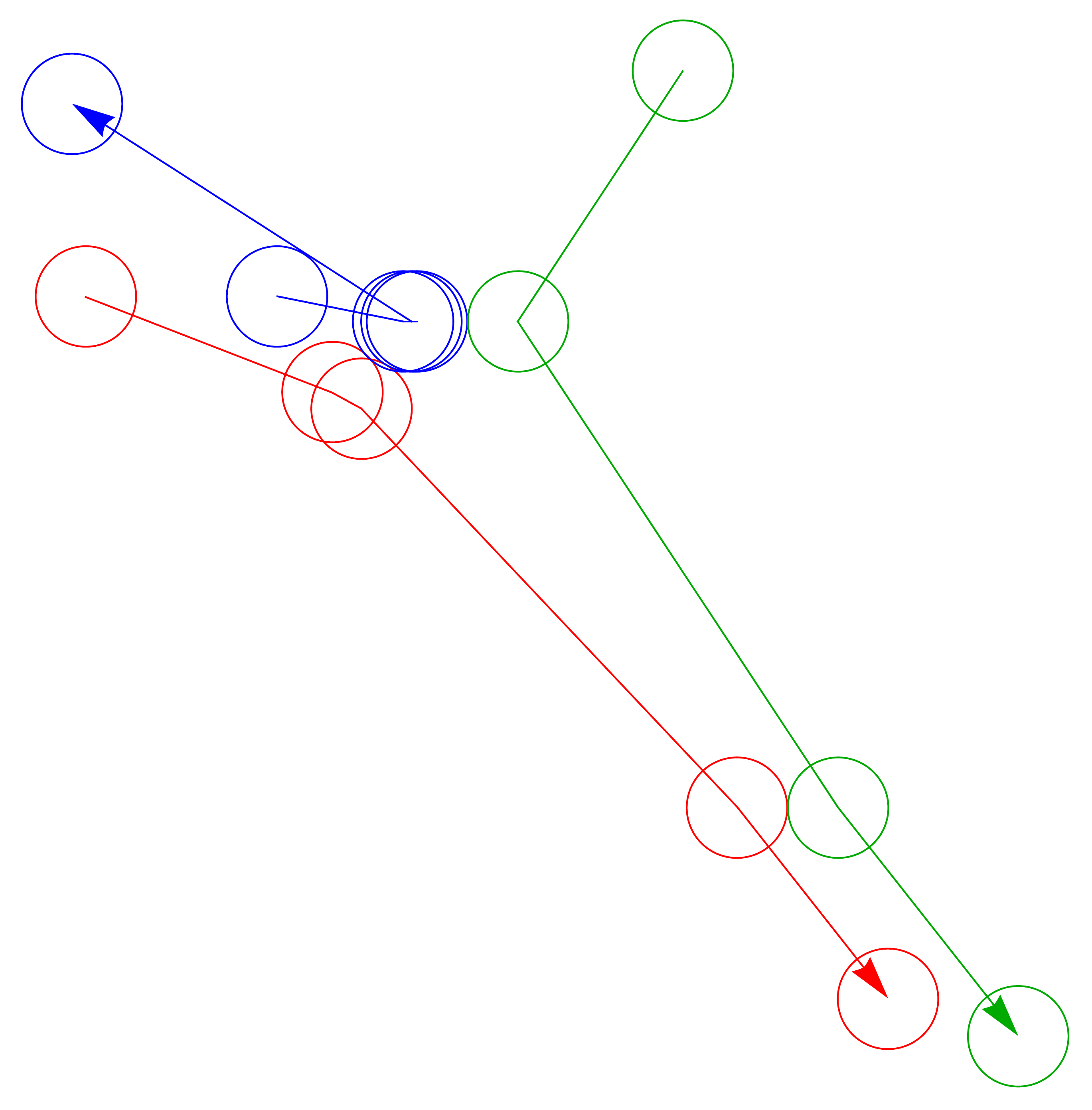}
\caption{
The Foch example---three discs in the plane colliding four times. First, the red disc hits the blue disc ``from behind.'' Then blue and green discs collide head on. Next the blue disc hits the red disc in a ricochet fashion. Finally, red and and green discs collide at a far away location.
The discs are shown at the initial positions given in \cite{MurCoh}, at the times when they collide, and at a time past the collisions. The vertical component of the blue disc velocity is zero between its first and third collisions.
}
\label{fig8}
\end{figure}

Our proof of Theorem \ref{n6.1} is based on a modification of the Foch example. Our version is ``conceptual'' in the sense that we can prove that for some initial conditions, three discs will collide four times without making any explicit numerical calculations. Nevertheless, the reader may find an explicit numerical example helpful; we present it in Remark \ref{n6.2} and Fig.~\ref{fig9}. 
\begin{remark}\label{n6.2}
We will describe some features of trajectories depicted in Fig.~\ref{fig9}.
The initial position of the red disc is at the origin. The initial velocity of the red disc is zero. 
The green disc hits the red disc with a great velocity. Then the  red disc slightly pushes the blue disc. The latter has been already moving in the SE direction and the slight push hardly changes its direction of motion.
After colliding with the red disc, the green disc is moving horizontally to the right, where it eventually hits the blue disc.
The fourth collision occurred before time 0, i.e., before the initial positions depicted in the figure. Specifically, when we reverse the direction of time, the blue disc will move in the NW direction with velocity $(-0.2, 0.05)$ from its initial position and it will hit the red disc. The initial position of the center of the green disc is distorted in the figure to help visualize the example. In fact, the first coordinate of the initial position of the center of the green disc is $-0.0086$; this would not be discernible from zero by the naked eye. Similarly, the last position of the center of the red disc is distorted. At the time of the third collision, i.e., the collision of the green and blue discs, the red disc is about 800 units away from the origin, in the direction that is much closer to the vertical  than the one shown in the figure. All other parts of the trajectory are depicted accurately. The discs are shown at the initial positions, at the times when they collide, and at a time past the collisions. 
The renditions of the blue disc at the initial time and the time of the first collision are so close to one another that they are visually indistinguishable.
The initial positions of the centers of red, blue and green discs are $(0,0), (1.9961, 0.5)$ and $(-0.0086, -3)$. The initial velocities of red, blue and green discs are $(0,0), (0.2, -0.05)$ and $(0.6622, 77)$. The  three positive collision times are $0.012987, 0.0193063$ and $10.4153$.
The  positions of the centers of red, blue and green discs at the last collision time $t=10.4153$, i.e., the time of the collision of blue and green discs, are $(-2.906, 800.963), (6.98517, -0.00268084)$ and $(7.08868, -2)$.
\end{remark}
\begin{figure} \includegraphics[width=0.9\linewidth]{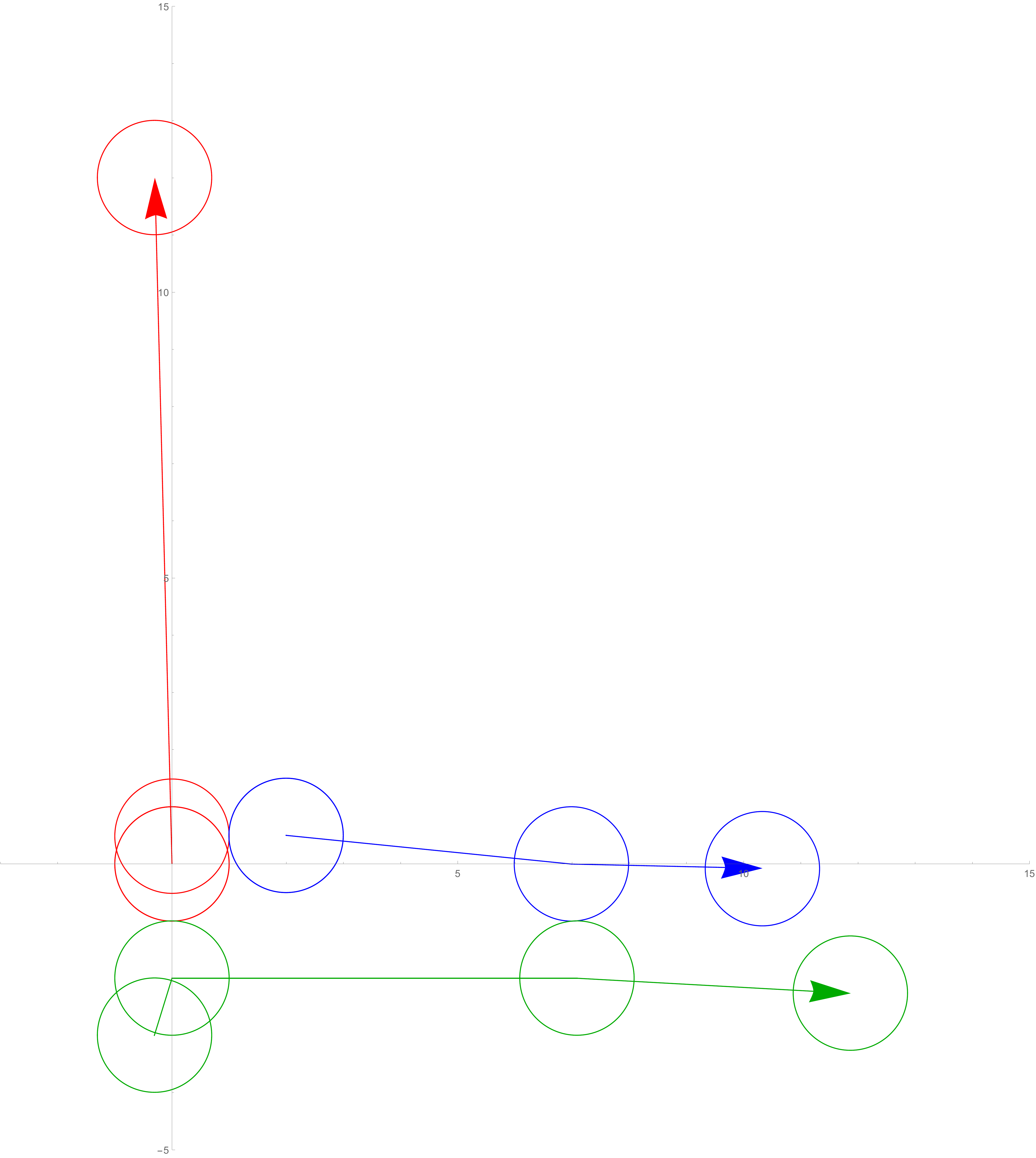}
\caption{
Three discs in the plane colliding four times. The first collision is not visible in the picture because it occurs at a negative time. See Remark \ref{n6.2} for the detailed description of the figure.
}
\label{fig9}
\end{figure}

\setcounter{step}{0}

\begin{proof}[Proof of Theorem \ref{n6.1}]
Note that it will suffice to prove the theorem for $d=2$.

\begin{step}\label{tt1}
First, we will consider only three discs. We will give a conceptual proof of Foch's result  that three discs may have four collisions.
Consider three discs with radii equal to 1, labeled $R,G$ and $B$, representing colors red, green and blue used in Fig.~\ref{fig9}.
Let $x^R(t) = (x^R_1(t), x^R_2(t))$ and $v^R(t) = (v^R_1(t), v^R_2(t))$ denote the position and velocity of disc $R$ at time $t$. Analogous notation will be used for discs $G$ and $B$.

We start by specifying only some of the initial conditions,
\begin{align*}
x^R(0) &= (0,0), \qquad v^R(0-) = (0,0),\\
x^G(0) & = (0,-2), \\
x^B_2(0) &= 0.5, \qquad v^B(0-) = (4,-1).
\end{align*}
We will choose $x^B_1(0)$ in the interval $(1.95, 2)$. It is elementary to check that if $x^B_1(0)$ is any number in the interval $(1.95, 2)$ then

(a) discs $R$ and $B$ do not intersect or touch at time $t=0$, and

(b) if we change the arrow of time, i.e., if the disc $B$ starts moving at time $t=0$ in the direction $(-4,1)$ and the disc $R$ remains static with the center at $(0,0)$ then the two discs will collide. 

Suppose for a moment that $x^B(0) = (2, 0.5)$ and $v^B(0-) = (4,-1)$.
Then the center of $B$ will cross the horizontal axis at time $t=0.5$ at the point $(4,0)$. 
We will now use the continuity of trajectories as functions of the initial conditions (see Remark \ref{re:ja2.1b}).
Find $\delta_1\in(0,0.05)$ so small that if $\delta \in(0,\delta_1)$, $x^B_2(\delta) = 0.5$,
$x^B_1(\delta) \in (2-\delta,2)$ and 
\begin{align}\label{ja1.1}
|v^B(\delta) - (4,-1)| \leq \delta
\end{align}
then the center of $B$ will cross the horizontal axis at a time $t_1 \in (0.4, 0.6)$ at a point $(z,0)$ with $z\in (3,5)$. 
We will choose a $\delta$ satisfying the above condition  and some more conditions later.  

Now suppose that $v^G(0+) = (w,0)$, where $w$ is chosen so that $x^G(t_1) = (z,-2)$. Then discs $B$ and $G$ will collide at time $t_1$. Since $t_1 \in (0.4, 0.6)$ and $z\in (3,5)$, we must have $w\geq 3 /0.6 = 5$.

At the moment of the collision the centers of $B$ and $G$ will lie on a vertical line. Hence,
after the collision, i.e., for $t>t_1$, the horizontal  velocities of discs $B$ and $G$ will be the same as before the collision and we will have 
\begin{align*}
v^B_1(t) & \leq 4 +\delta< 4.05, \qquad v^B_2(t) = 0,\\
v^G_1(t) & =w\geq 5, \qquad v^G_2(t) \geq -1 -\delta > -1.05.
\end{align*}
This implies that for large $t$, the line passing through $x^B(t)$ and $x^G(t)$ will intersect the horizontal axis at a point $(u,0)$ with $u \leq 4.05 t$ and its slope will be greater than $-2$ because  $- 1.05/(5-4.05) > -2$.

Find $r_1>0$ so large and $\eps_1>0$ so small that if $r\geq r_1$, $\eps\in(0,\eps_1)$, $x^R_1(\eps)=0$, $x^R_2(\eps) \in[0,0.5]$ and
\begin{align}\label{ja1.2}
|v^R(\eps) - (0,r)| \leq \eps_1
\end{align}
then for large $t$, the center of $R$ will lie above any line   
which intersects the horizontal axis at a point $(u,0)$ with $u \leq 4.05 t$ and has a slope  greater than $ -2$. Note that if this holds true then there must be a time such that the centers of $R, G$ and $B$ are aligned.

Suppose for a moment that $x^B(0) = (2-\delta_1, 0.5)$ and $v^B(0-) = (4,-1)$ and let $t_3>0$ be so small that the point $(1,0.5)$ belongs to the interior of $B$ for all times in $[0,t_3]$.

Suppose that $r_1$ and $\eps_1$ satisfy the above conditions and, in addition, $r_1 > 0.5/t_3$. If $v^R(0) = (0,r)$ and $r\geq r_1$, and we ignore disc $B$ then 
we will have $x^R_2(t_2) = 0.5$ for some $t_2 < t_3$.
We now let $v^G_2(0-) = r = 2r_1$. At time $t=0$, discs $R$ and $G$ will collide and we will have $v^R(0+) = (0,r)$. 

If we let $x^B(0) = (2-\delta_1,0.5)$ then discs $R$ and $B$ will collide. If, on the other hand, we let $x^B(0) = (2,0.5)$ then  $R$ and $B$ will not collide. Hence, there exists $\delta \in (0, \delta_1)$ such that if
$x^B(0) = (2-\delta,0.5)$ then discs $R$ and $B$ will collide and the collision will be so close to the ``grazing'' collision that conditions \eqref{ja1.1} and \eqref{ja1.2} will be satisfied. At this point we let $v^G_1(0-) = w$ as described earlier in the proof and we note that we also have $v^G_1(0+) = w$. 

We have constructed trajectories of $R,G$ and $B$ in such a way that the discs have three collisions at times $t\geq 0$, one collision at a time $t<0$ (see point (b) earlier in the proof concerning the last claim), and their centers are aligned at a certain time following the four collisions.

\end{step}

\begin{step}\label{tt2}

Fix an arbitrary integer $m>0$. Suppose that at time $t=0$, the centers of discs $R,G$ and $B$ lie on a straight line $L$ in this order and they do not touch each other. Moreover, suppose that the speeds of the three discs are zero. It is elementary to find initial positions and velocities  of discs $A_1, A_2, \dots, A_m$ with radii 1 so that their centers lie on $L$ at time $t=0$, initial velocities are parallel to $L$, and the evolution of the system is the following. Disc $A_1$ will hit $R$, then $R$ will hit $G$ and then $G$ will hit $B$. Next, $A_2$ will hit $A_1$, $A_1$ will hit $R$, then $R$ will hit $G$ and then $G$ will hit $B$. At the $k$-th stage, $A_k$ will hit $A_{k-1}$, $A_{k-1}$ will hit $A_{k-2}$, $\dots$, $A_2$ will hit $A_1$, $A_1$ will hit $R$, then $R$ will hit $G$ and then $G$ will hit $B$.
The graphical representation of the system described in Example \ref{oc18.2} is probably the easiest way to see that the evolution described above exists. 

The total  number of collisions will be $3+4+\dots + (m+2) = (m^2 + 5m)/2 $.
By continuity of trajectories as functions of the initial conditions (see Remark \ref{re:ja2.1b}), we do not have to assume that the initial speeds of $R,G$ and $B$ are zero. It is enough to assume that their speeds are sufficiently small and we will still have the same number of collisions. But multiplying all velocities in the system by the same strictly positive scalar does not change the number of collisions, so the initial speeds 
of $R,G$ and $B$ can be arbitrary.

\end{step}

\begin{step}\label{tt3}

We now combine the two steps. According to Step \ref{tt1}, $R,G$ and $B$  have four collisions and at a certain time following the collisions, their centers will be aligned. At this time, we ``add'' discs $A_1, \dots A_m$ to the system. The total number of collisions will be  $  (m^2 + 5m)/2 + 4$. The total number of discs will be  $n=3+m$ so the total number of collisions may be expressed as $  (m^2 + 5m)/2 + 4 = n(n-1)/2 +1$.
\end{step}
\end{proof}

\section{Appendix}
\label{se:app}

This section contains some technical results that are used in the proof of the main theorem of this article.

\begin{lemma}
\label{le:extension}
Let $T>0$ be fixed, and assume that $D  f_k\to D  f$ and $D g_k\to D g$ in $\calD [0,R]$, for all $R<T$. Suppose that there exist $\delta>0$ and a sequence $t_k\to T$ such that 
\begin{enumerate}[(i)]
\item $D  g$ and $D  f$ are constant in $[T-2\delta,T)$, and in $(T,T+2\delta]$.
\item $D  g_k$ and $D  f_k$ are constant in $[T-\delta,t_k)$, and in $(t_k,T+\delta]$.
\item There is a sequence $\alpha_k\to 0$ such that
\begin{align*}
| D  f_k(t_k) - D  f(T) | \leq  | D  g_k(t_k-) - D  g(T-) | +  | D  f_k(t_k-) - D  f(T-) |+\alpha_k.
\end{align*}
\end{enumerate}
Then, $D  f_k\to D  f$ in $\calD [0,T+\delta]$.
\end{lemma}
\begin{proof}
Recall notation from Section \ref{Skor}. Since $D  f_k\to D  f$, there exist $\beta_k$ converging to $0$ and $\wt\lambda_k\in \Lambda_{T-\delta}$  such that $||\wt\lambda_k||_{T-\delta} \leq \beta_k$ and 
\begin{align*}
\sup_{0 \leq u \leq T-\delta}
| D  f_k(\wt\lambda_k(u)) - D  f(u)|\leq \beta_k.
\end{align*}
Define $\lambda_k$ as follows: 
\begin{align*}
\lambda_k(x) = 
\begin{cases}
\wt\lambda_k(x), & x\leq T-\delta, \\
x-T+t_k + \rpar{\frac{x-T}{\delta}}^2 (T-t_k), &  T-\delta < x\leq T+\delta.
\end{cases}
\end{align*}
Note that $\lambda_k(0)=0$, $\lambda_k(T-\delta)=T-\delta$, $\lambda_k(T) =t_k$, and $\lambda_k(T+\delta) = T+\delta$. We have $\frac \prt{\prt x} \lambda_k(x) = 1 + O(\frac{1}{\delta} |T-t_k|)$ for $x\in [T-\delta,T+\delta]$. These observations imply that $||\lambda_k||_{T+\delta} \leq\max\kpar{||\wt\lambda_k||_{T-\delta}, C_1\delta^{-1} |T-t_k|}$ for large enough $k$, and some universal $C_1$.

By $(i)$ and $(ii)$, it follows that for all  $u \in [T,T+\delta]$
\begin{align*}
| D & f_k(\lambda_k(u)) - D  f(u)| = | D  f_k( t_k ) - D  f(T)| \\
&\leq  | D  g_k(t_k-) - D  g(T-) | +  | D  f_k(t_k-) - D  f(T-) | +\alpha_k \\
&=  | D  g_k(T-\delta) - D  g(T-\delta) | +  | D  f_k(T-\delta) - D  f(T-\delta) | + \alpha_k.
\end{align*}
For all $u\in [T-\delta,T)$,
\begin{align*}
| D  f_k(\lambda_k(u)) - D  f(u)| &= | D  f_k(T-\delta) - D  f(T-\delta) | .
\end{align*}
Therefore,
\begin{align*}
\dist_0^{T+\delta } &(D  f_k,D  f) 
\leq \sup_{0 \leq u \leq T+\delta}
| D  f_k(\lambda_k(u)) - D  f(u)|
+ ||\lambda_k||_{T+\delta}\\
&\leq
\sup_{0 \leq u \leq T-\delta}
| D  f_k(\lambda_k(u)) - D  f(u)|
+ ||\lambda_k||_{T+\delta}\\
&\qquad + \sup_{T-\delta \leq u \leq T}
| D  f_k(\lambda_k(u)) - D  f(u)|
+ \sup_{T \leq u \leq T+\delta}
| D  f_k(\lambda_k(u)) - D  f(u)|\\
&\leq
\sup_{0 \leq u \leq T-\delta}
| D  f_k(\wt\lambda_k(u)) - D  f(u)|
+ ||\wt\lambda_k||_{T-\delta}+ C_1\delta^{-1} |T-t_k|\\
&\qquad + | D  f_k(T-\delta) - D  f(T-\delta) |\\
&\qquad+ | D  g_k(T-\delta) - D  g(T-\delta) | +  | D  f_k(T-\delta) - D  f(T-\delta) | + \alpha_k\\
&\leq
2\beta_k + C_1\delta^{-1} |T-t_k|\\
&\qquad+ | D  g_k(T-\delta) - D  g(T-\delta) | + 2 | D  f_k(T-\delta) - D  f(T-\delta) | + \alpha_k.
\end{align*}
For a fixed $\delta>0$, 
 when $k$ goes to infinity, 
$\beta_k \to 0$, $C_1\delta^{-1} |T-t_k| \to 0$ and $\alpha_k \to 0$.
Since $T-\delta$ is a continuity point of $D g$ and $Df$ (by assumption $(i)$), and Skorohod convergence implies pointwise convergence at continuity points,
we obtain
\begin{align*}
\lim_{k\to\infty}\left(
| D  g_k(T-\delta) - D  g(T-\delta) | + 2 | D  f_k(T-\delta) - D  f(T-\delta) |\right) = 0.
\end{align*}
Hence, $\lim_{k\to \infty} \dist_0^{T+\delta } (D  f_k,D  f) =0$.
\end{proof}

\begin{corollary}
\label{co:extension}
Let $T>0$ be fixed and assume that $D  f_k\to D  f$ in $\calD [0,R]$, for all $R<T$. Assume there is $\alpha >0$ such that $D  f_k$ and $D  f$ are constant in $[T-\alpha,T+\alpha]$. Then $D  f_k\to D  f$ in $\calD [0,T+\alpha]$.
\end{corollary}
\begin{proof}
Apply Lemma \ref{le:extension} with $\alpha =\delta$, $g_k=g=0$, $t_k=T$, and $\alpha_k=0$.
\end{proof}

\begin{lemma}
\label{le:sk_unif}

Let $\kpar{\bS^k}$ be a sequence of families of $2n_1+1$ functions, as in \eqref{f10.1}, $\kpar{\eps_k}$ be a sequence of positive numbers, and let $\{\bX^k\}$ be  the corresponding sequence as in \eqref{eq:bX}. Assume that $|D\bS^k(0)|=1$ and
\begin{enumerate}[(i)]
\item $\bX^k(0)$ converge to $\bZ(0)$ and $\{\bZ(t), t\geq 0\}$ is defined as in \eqref{eq:nz}; and
\item $D  \bX^k$ converge to $D \bZ$ in the Skorohod space $\calD[0,R]$ for all $0<R<T$.
\end{enumerate}
Then $\bX^k$ converge uniformly to $\bZ$ in $[0,T]$.
\end{lemma}

\begin{proof}
By the right continuity of each $D \bX^k$ and $D \bZ$, the Fundamental Theorem of Calculus implies that,
\begin{align}\label{f10.2}
\sup_{0 \leq t \leq T} | \bX^k(t) - \bZ(t) |
&\leq | \bX^k(0) - \bZ(0) | + \int_0^T | D \bX^k(u) - D \bZ(u) | \ du.
\end{align}
Since Skorohod convergence implies pointwise convergence at continuity points, assumption $(ii)$ implies that $D \bX^k$ converge to $D \bZ$ almost everywhere in $[0,T]$. 
For all $u\geq 0$, $| D \bX^k(u) | \leq 3$ by \eqref{f10.3} and $| D \bZ(u) |\leq n+2$ by Remark \ref{re:vz} (ii). Hence the lemma follows from \eqref{f10.2}, by the Dominated Convergence Theorem.
\end{proof}

\begin{lemma}
\label{le:gap_neg}
Let  $\bS$ be a family of $2n_1+m$ functions satisfying $(\eps,\rho)$ initial conditions. Then, for all $t\geq 0$, $ i \geq 1$,  we have
\begin{align}
\label{eq:gap_negA}
-2\eps (1+t)^2 \leq \calX^{A,\eps}_{i}(t) - \calX^{A,\eps}_{i+1}(t), \\
\label{eq:gap_negB}
-2\eps (1+t)^2 \leq \calX^{B,\eps}_{i}(t) - \calX^{B,\eps}_{i-1}(t), \\
\label{eq:gap_negC}
-2\eps (1+t)^2 \leq \calX^{C,\eps}_{i}(t) - \calX^{C,\eps}_{i-1}(t).
\end{align}
\end{lemma}

\begin{proof}
We will prove \eqref{eq:gap_negA}. Completely analogous arguments show the other two inequalities. The distance between centers of balls is at least 2 at any time, thus
\begin{align*}
4 &\leq [w_0\cdot ( a_i(\eps t) - a_{i+1}(\eps t))]^2 + [u_0\cdot( a_i(\eps t) - a_{i+1}(\eps t))]^2 .
\end{align*}
A straightforward computation using the assumption that the speed of any disc is at most one and the initial condition \eqref{eq:ic_u} yield $|u_0 \cdot a_i(\eps t)| \leq \eps(1+t)$ for all $i$. We have  $ w_0 \cdot ( a_i(\eps t) - a_{i+1}(\eps t))  = 2+\eps (\calX^{A,\eps}_i(t) - \calX^{A,\eps}_{i+1}(t)) $. Therefore, for all  $i$,
\begin{align*}
4 &\leq ( 2 + \eps ( \calX^{A,\eps}_i(t) - \calX^{A,\eps}_{i+1}(t) )^2 + 4\eps^2(1+t)^2.
\end{align*}
We use the assumption that the disc speeds are bounded by 1 again, and the initial conditions \eqref{eq:ic_gap}, to arrive at
\begin{align*}
0 &\leq 4 \eps ( \calX^{A,\eps}_i(t) - \calX^{A,\eps}_{i+1}(t) )+ \eps^2(1+2t)^2+ 4\eps^2 ( 1+ t )^2 \\
&\leq  4 \eps ( \calX^{A,\eps}_i(t) - \calX^{A,\eps}_{i+1}(t) )+ 8\eps^2 ( 1+ t )^2,
\end{align*}
from which \eqref{eq:gap_negA}  follows.
\end{proof}

\begin{lemma}
\label{le:zero}
Let $\bS^k$ and $T^*$ be as in the proof of Proposition \ref{pr:s1} .  Suppose that $\calZ^B_j(T^*)=\calZ^B_{j+1}(T^*)$ for some $j\geq 0$. Then there exists $\delta_0>0$ such that for all $\delta \in (0,\delta_0]$, there is $k_0\geq 0$ such that for all $k\geq k_0$, there is exactly one discontinuity of $D  (\calX^{B,\eps_k}_{j+1}-\calX^{B,\eps_k}_{j})$ in $[T^*-\delta,T^*+\delta]$. An analogous statement holds if $\calZ^C_j(T^*)=\calZ^C_{j+1}(T^*)$, or if $\calZ^A_i(T^*)=\calZ^A_{i+1}(T^*)$ for $i\geq 1$.
\end{lemma}
\begin{proof}
The assumption that $\calZ^B_j(T^*)=\calZ^B_{j+1}(T^*)$ implies that there is a discontinuity of both $D  \calZ^B_j$ and $D  \calZ^B_{j+1}$ at $T^*$.
There are finitely many of these discontinuities, so there is $\delta_1>0$ such that $D (\calZ^B_{j+1}-\calZ^B_{j})$ is constant in $(T^*-\delta_1,T^*)$, and so $\calZ^B_{j+1} > \calZ^B_{j}$ in this interval. Since  $\calZ^B_j(T^*)=\calZ^B_{j+1}(T^*)$, it follows that $D  (\calZ^B_{j+1} - \calZ^B_{j})(T^*-) <0$.

Fix $\delta\in(0,\delta_1)$. If for infinitely many $k$ we have that  $D  ( \calX^{B,\eps_k}_{j+1} - \calX^{B,\eps_k}_{j})$ does not have a discontinuity in $[T^*-\delta,T^*+\delta]$, then $D  (\calX^{B,\eps_k}_{j+1}-\calX^{B,\eps_k}_{j})$ is constant in $(T^*-\delta,T^*+\delta)$ for such values of $k$. 
We will assume without loss of generality that this claim holds for all $k$,  because otherwise we can pass to a subsequence. Fix $0< r <\delta$ and $t \in (T^*,T^*+\delta)$. It follows from Lemma \ref{le:gap_neg} that,
\begin{align}\label{f11.1}
-2\eps_k(1+t)^2 &\leq \calX^{B,\eps_k}_{j+1}(t)-\calX^{B,\eps_k}_{j}(t) \\
&= (\calX^{B,\eps_k}_{j+1}-\calX^{B,\eps_k}_{j})(T^*-r) + D  (\calX^{B,\eps_k}_{j+1}-\calX^{B,\eps_k}_{j}) (T^*-r) \cdot (t-T^*+r).\notag
\end{align}
Since $D  (\calZ^B_{j+1}-\calZ^B_{j})$ is constant in $(T^*-\delta_0,T^*)$,  it is continuous at $T^*-r$, and, therefore, $D  (\calX^{B,\eps_k}_{j+1}-\calX^{B,\eps_k}_{j})(T^*-r)$ converge to $D  (\calZ^B_{j+1}-\calZ^B_{j})(T^*-r) = D  (\calZ^B_{j+1}-\calZ^B_{j})(T^*-)$ because Skorokhod convergence implies pointwise convergence at continuity points. Taking $k$ to infinity in \eqref{f11.1} and using Lemma \ref{le:sk_unif}, we obtain
\begin{align*}
0 \leq (\calZ^B_{j+1}-\calZ^B_{j})(T^*-r) + D (\calZ^B_{j+1}-\calZ^B_{j})(T^*-) \cdot (t-T^*-r).
\end{align*}
Taking $r$ to zero yields $0\leq D (\calZ^B_{j+1}-\calZ^B_{j})(T^*-)$, a contradiction. This shows that there is $k_0\geq 0$ such that for all $k\geq k_0$ there is at least one discontinuity of $D  (\calX^{B,\eps_k}_{j+1}-\calX^{B,\eps_k}_{j})$ in $[T^*-\delta,T^*+\delta]$, that is, $B_j$ or $B_{j+1}$ participate in a collision on this interval.

Recall from
Remark \ref{re:vz} (vii) that there are no ``simultaneous collisions'' among components of $\bZ$. It follows that there is $\alpha >0$ such that $|\calZ^B_j(T^*) - \calZ^D_i(T^*)| >4\alpha$ for every function of the form $\calZ^D_i$, $D=A,B,C$, $i\geq 0$,  except for  $\calZ^B_{j+1}$. It follows from Lemma \ref{le:sk_unif} that  for $k$ large enough, we have that $| \calX^{B,\eps_k}_j(t) - \calX^{D,\eps_k}_i(t) | \geq 3\alpha$ for all $t \in [T^*-\alpha,T^*]$, $D=A,B,C$, and $i\geq 0$,  except for  $\calZ^B_{j+1}$. 
Since the speeds of $ \calX^{B,\eps_k}_j(t) $ and $ \calX^{D,\eps_k}_i(t) $ are bounded by 1, we obtain the bound $| \calX^{B,\eps_k}_j(t) - \calX^{D,\eps_k}_i(t) | \geq \alpha$ for all $t \in [T^*-\alpha,T^*+\alpha]$.
A similar reasoning shows that $B_{j+1}$ can only collide with $B_j$ in this interval. The discs $B_{j+1}$ and $B_j$ can collide only once in $ [T^*-\alpha,T^*+\alpha]$ because in order to collide twice, one of them would have to hit some other disc. We conclude that the lemma holds with $\delta_0 = \min\kpar{\delta_1,\alpha}/2$.

It is easy to see that the same argument applies to discs in families $A$ and $C$.
\end{proof}

\begin{lemma}
\label{le:velpm}
Let $\bS$ be a family of functions satisfying $(\eps,\rho)$ initial conditions at $t=0$. For $j\geq 0$, if there is a collision between discs $B_j$ and $B_{j+1}$ at time $\eps t$ then, 
\begin{align}
\label{eq:velpm01}
D  \calX^{B,\eps}_j (t) &= D  \calX^{B,\eps}_{j+1} (t-) + O(\eps(1+t)), \\
\label{eq:velpm03}
D  \calX^{B,\eps}_{j+1} (t) &= D  \calX^{B,\eps}_{j} (t-) + O(\eps(1+t)).
\end{align}
Analogous estimates hold if the collision is between $A_j$ and $A_{j+1}$, or $C_j$ and $C_{j+1}$. 

In case $j=0$, we have
\begin{align}
\label{eq:velpm05}
D  \calX^{C,\eps}_0 (t)  &= D  \calX^{C,\eps}_0 (t-) - \frac12 D  \calX^{B,\eps}_1(t-) +  \frac12 D  \calX^{B,\eps}_0(t-) + O(\eps(1+t)).
\end{align}
\end{lemma}

\begin{remark}\label{f12.1}
If we set  $B_0=C_0=A_0$,  $b_0(t)=a_1(t)$, and $c_0(t)=a_1(t)$ then
 Lemma \ref{le:velpm} also covers collisions between $A_1$ and $B_1$, and collisions between $A_1$ and $C_1$.
\end{remark}

\begin{proof}[Proof of Lemma \ref{le:velpm}]
We will prove \eqref{eq:velpm01}-\eqref{eq:velpm03} only for the family $B$. The proof is completely analogous for the families $A$ and $C$.

Using \eqref{oc2.3} at time $\eps t$, we obtain
\begin{align}\label{f11.4}
D b_j(\eps t) &= D b_j(\eps t-) - 
\frac 1 4 \spar{ D (b_j - b_{j+1})(\eps t-)\cdot (b_j-b_{j+1})(\eps t)} 
(b_j-b_{j+1})(\eps t).
\end{align}
It follows from \eqref{eq:ic_u}  and the fact that the speed of each disc is at most 1 at any time that
\begin{align}\label{f11.2}
| u_1\cdot (b_j-b_{j+1})(\eps t) | &\leq |u_1\cdot (b_j-b_{j+1})(0)| + \int_0^{\eps t} | u_1\cdot D  (b_j-b_{j+1})(u) |\ du \\
&\leq 2\eps + 2\eps t = O(\eps(1+t)). \notag
\end{align}
Since  $|D  ( \calX^{B,\eps}_{j+1} - \calX^{B,\eps}_{j} )(t)|\leq 2$ at all times, we have $| ( \calX^{B,\eps}_{j+1} - \calX^{B,\eps}_{j} ) (t)| \leq 1 + 2t$.
This and \eqref{f11.2} imply that
\begin{align}\label{f11.3}
(b_j-b_{j+1})(\eps t) &= P_{w_1}(b_j-b_{j+1})(\eps t) + P_{u_1}(b_j-b_{j+1})(\eps t) \\
&= -\rpar{2+\eps ( \calX^{B,\eps}_{j+1} - \calX^{B,\eps}_{j} )(t)}w_1 + O(\eps(1+t))\notag \\
&= -2w_1 + O(\eps(1+t)).\notag
\end{align}
Recall that $D \calX^{B,\eps}_j(t) = w_1\cdot D  b_j(\eps t)$ for $j\geq 1$.
This implies that  
$D (b_j - b_{j+1})(\eps t-)\cdot w_1
=- D  ( \calX^{B,\eps}_{j+1} - \calX^{B,\eps}_{j} )(t-) $.
We combine this observation and \eqref{f11.3} to obtain
\begin{align*}
D (b_j - b_{j+1})(\eps t-)\cdot (b_j-b_{j+1})(\eps t ) 
&= 2D  ( \calX^{B,\eps}_{j+1} - \calX^{B,\eps}_{j} )(t-) + O(\eps(1+t)).
\end{align*}
This, \eqref{f11.4} and \eqref{f11.3} yield
\begin{align}
\label{d22.06}
D b_j(\eps t) 
&= D b_j(\eps t-) + w_1 D ( \calX^{B,\eps}_{j+1} - \calX^{B,\eps}_{j} )(t-) + O(\eps(1+t)) .
\end{align}
Next we apply  the scalar product with $w_1$, to see that
\begin{align}\label{f11.5}
D \calX^{B,\eps}_j(t) &=
D b_j(\eps t) \cdot w_1\\
&= D b_j(\eps t-)\cdot w_1 +w_1\cdot w_1 D ( \calX^{B,\eps}_{j+1} - \calX^{B,\eps}_{j} )(t-) + O(\eps(1+t))\cdot w_1\notag \\
& = D \calX^{B,\eps}_j(t-) + D ( \calX^{B,\eps}_{j+1} - \calX^{B,\eps}_{j} )(t-) + O(\eps(1+t))\notag\\
& =  D  \calX^{B,\eps}_{j+1}(t-) + O(\eps(1+t)),\notag
\end{align}
which is \eqref{eq:velpm01}. The estimate \eqref{eq:velpm03} follows from \eqref{eq:velpm01} and conservation of momentum at the collision time.

Set $j=0$ in \eqref{d22.06}, recall Remark \ref{f12.1} and apply conservation of momentum at the collision time to deduce that
\begin{align*}
D  a_1(\eps t) = D  a_1(\eps t-)  + w_1 D  ( \calX^{B,\eps}_{1} - \calX^{B,\eps}_{0} ) (t-) + O(\eps(1+t)).
\end{align*}
Taking the scalar product with $w_2$ and computing as in \eqref{f11.5}, we obtain \eqref{eq:velpm05}.
\end{proof}

\section{Acknowledgments}
We are grateful to Jayadev Athreya and Jaime San Martin  for very helpful advice.
We thank the referee for the suggestions for improved presentation of our results.

\bibliographystyle{alpha}
\bibliography{hardlow}

\end{document}